\documentclass{article}
\usepackage{amsmath}
\usepackage{amsfonts}
\usepackage{graphicx}

\newcommand{\betless}{\noalign{\vskip3pt plus 3pt minus 1pt}}
\newcommand{\bet}{\noalign{\vskip6pt plus 3pt minus 1pt}}

\renewcommand{\O }{\Omega }

\renewcommand{\lor }{\longrightarrow}

\newtheorem{theorem}{Theorem}
\newtheorem{lemma}{Lemma}
\newcounter{remark}
\def\theremark {\arabic{remark}}
\newenvironment{remark}{\refstepcounter{remark}\par\noindent{\bf Remark\ \theremark}\ }{\par}
\newtheorem{Proof}{Proof}

\newenvironment{proof}{\begin{Proof}\rm}{\hfill $\Box$ \end{Proof}}

\title{A posteriori error estimations for mixed finite-element approximations to the Navier-Stokes equations}
\author{Javier de Frutos\thanks{Departamento de Matem\'{a}tica Aplicada,
Universidad de Valladolid. Spain. Research supported by Spanish MEC
under grant MTM2007-60528 (frutos@mac.uva.es)} \and Bosco
Garc\'{\i}a-Archilla\thanks{Departamento de Matem\'{a}tica Aplicada
II, Universidad de Sevilla, Sevilla, Spain. Research supported by
Spanish MEC under grant MTM2009-07849 (bosco@esi.us.es)}
  \and Julia Novo\thanks{Departamento de
Matem\'aticas, Universidad Aut\'onoma de Madrid, Instituto de
Ciencias Matem\'aticas CSIC-UAM-UC3M-UCM, Spain. Research supported
by Spanish MEC under grant MTM2007-60528 (julia.novo@uam.es)}}
\begin{document}
\maketitle

\begin{abstract} A posteriori estimates for mixed finite element discretizations of the Navier-Stokes equations
are derived. We show that the task of estimating the error in the
evolutionary Navier-Stokes equations can be reduced to the
estimation of the error in a steady Stokes problem. As a
consequence, any available procedure to estimate the error in a
Stokes problem can be used to estimate the error in the nonlinear
evolutionary problem. A practical procedure to estimate the error
based on the so-called postprocessed approximation is also
considered. Both the semidiscrete (in space)  and the  fully
discrete cases are analyzed. Some numerical experiments are
provided.
\end{abstract}

\section{Introduction}
\label{sec:1}
We consider
the incompressible Navier--Stokes equations
\begin{eqnarray}
\label{onetwo}
u_t -\Delta u + (u\cdot\nabla)u + \nabla p &=& f,\\
\bet
{\rm div}(u)&=&0,\nonumber
\end{eqnarray}
in a bounded domain $\Omega\subset {\mathbb R}^d$ ($d=2,3$) with a smooth
boundary subject to homogeneous Dirichlet boundary conditions $u=0$
on~$\partial\Omega$. In~(\ref{onetwo}), $u$ is the velocity
field, $p$ the pressure, and~$f$ a given force field. For simplicity in the exposition
 we assume, as in \cite{Bause}, \cite{heyran0}, \cite{heyran2}, \cite{HR-IV}, \cite{kara-makri}, that
the fluid density and viscosity have been normalized by an adequate change
of scale in space and time.

Let  $u_h$ and $p_h$ be the semi-discrete (in space) mixed finite element (MFE) approximations to the velocity~$u$
and pressure~$p$, respectively, solution of~(\ref{onetwo}) corresponding
to a given initial condition
\begin{eqnarray}\label{ic}
u(\cdot,0)=u_0.
\end{eqnarray}
We study the a posteriori error estimation of these approximations in the $L^2$ and $H^1$ norm for the velocity and in the $L^2/{\Bbb R}$ norm for the pressure. To
do this for a given time~${t^*}>0$, we consider the
solution ($\tilde u$, $\tilde p$) of the Stokes problem
\begin{equation}
\begin{array}{rcl}
\left.\begin{array}{r@{}}
-\Delta \tilde u +\nabla \tilde p=f-\frac{d}{dt}u_h({t^*})
-(u_h({t^*})\cdot\nabla) u_h({t^*})
\\
\betless
{\rm div}(\tilde u)=0
\end{array}\right\}&\qquad \hbox{\rm in~$\Omega$},
\\
\bet
&\hspace*{-41pt}\tilde u=0, \quad \hbox{~~~~~~~~~~\,\rm on~$\partial\Omega$}.
\end{array}
\label{eq:stokes}
\end{equation}
We prove that $\tilde u$ and $\tilde p$ are approximations to~$u$ and~$p$ whose errors decay by a factor of~$h\left|\log(h)\right|$ faster than those of $u_h$ and~$p_h$ ($h$ being the mesh size).
 As a consequence,
the quantities $\tilde u-u_h$ and~$\tilde p-p_h$, are asymptotically
exact indicators of the errors
$u-u_h$ and $p-p_h$ in the Navier-Stokes
problem~(\ref{onetwo})--(\ref{ic}).

Furthermore, the key observation in the present paper is that
($u_h,p_h$) is also the MFE approximation to the solution $(\tilde
u,\tilde p)$ of the Stokes problem (\ref{eq:stokes}). Consequently,
any available procedure to a posteriori estimate the errors in a
Stokes problem can be used to estimate the errors $\tilde u -u_h$
and~$\tilde p-p_h$ which, as mentioned above, coincide
asymptotically with the errors~$u-u_h$ and $p-p_h$ in the
evolutionary NS equations. Many references address the question of
estimating the error in a Stokes problem, see for example
\cite{Ains-oden}, \cite{Bank-Welfert1},  \cite{Bank-Welfert2},
\cite{Jin}, \cite{Kweon}, \cite{Russo}, \cite{Verfurth1} and the
references therein. In this paper we prove that any efficient or
asymptotically exact estimator of the error in the MFE approximation
$(u_h,p_h)$ to the solution of the {\it steady\/} Stokes problem
(\ref{eq:stokes}) is also an efficient or asymptotically exact
estimator, respectively, of the error in the MFE approximation
$(u_h,p_h)$ to the solution of~the {\it evolutionary\/}
Navier-Stokes equations (\ref{onetwo})--(\ref{ic}).

For the analysis in the present paper we do not assume to have
more than second-order spatial derivatives bounded in $L^2(\Omega)^d$ up to initial time $t=0$, since demanding further regularity requires the data to satisfy nonlocal
compatibility conditions unlikely to be fulfilled in practical situations
\cite{heyran0}, \cite{heyran2}. The analysis of the errors $u-\tilde u$ and
$p-\tilde p$ follows closely \cite{jbj_regularity} where MFE approximations
to the Stokes problem (\ref{eq:stokes}) (the so-called postprocessed approximations) are considered with the
aim of getting improved approximations to the solution of (\ref{onetwo})--(\ref{ic}) at any fixed time
$t^*>0$. In this paper we will also refer to ($\tilde u$, $\tilde p$) as postprocessed approximations although they are of course not
computable in practice and they are only considered for the analysis of a posteriori
error estimators. The postprocessed approximations to the Navier-Stokes equations were
first developed for spectral methods  in \cite{Bosco-Julia-Titi}, \cite{Bosco-Julia-Titi-2}, \cite{novo3}, \cite{margolin-titi-wynne} and
also developed for MFE methods for the Navier-Stokes equations  in \cite{bjj}, \cite{bbj}, \cite{jbj_regularity}.

For the sake of completeness, in the present paper we also analyze the
use of the computable postprocessed approximations of~\cite{jbj_regularity}
for a posteriori error estimation.
The use of this kind of postprocessing technique to get a posteriori
error estimations has been studied in  \cite{Javier-Julia-estiman},
\cite{Javier-Julia-estiman-2} and \cite{Javier-Julia-estiman-3} for
nonlinear parabolic equations excluding the Navier-Stokes equations.
We refer also to \cite{kara-makri} where the so-called Stokes
reconstruction is used to a posteriori estimate the errors of the
semi-discrete in space approximations to a linear time-dependent
Stokes problem. We remark that the Stokes reconstruction of
\cite{kara-makri} is exactly the postprocessing approximation
($\tilde u,\tilde p$) in the particular case of a linear model.

In the second part of the paper we consider a posteriori error
estimations for the fully discrete MFE approximations $U_h^n\approx
u_h(t_n)$ and $P_h^n\approx p_h(t_n)$,  ($t_n=t_{n-1}\Delta t_{n-1}$ for
$n=1,2,\ldots,N$) obtained by integrating in time with
either the backward Euler method or the two-step backward
differentiation formula (BDF).  For this purpose, we define  a
Stokes problem similar to (\ref{eq:stokes}) but with the
right-hand-side depending now on the fully discrete MFE
approximation~$U_h^n$ (problem~(\ref{posth0n})--(\ref{posth1n}) in
Section~\ref{sec:4} below). We will call time-discrete postprocessed
approximation to the solution $(\widetilde U^n, \widetilde P^n)$ of
this new Stokes problem. As before, $(\widetilde U^n, \widetilde
P^n)$ is not computable in practice and it is only considered for
the analysis of a posteriori error estimation.

Observe that in the fully discrete case (which is the case in actual
computations) the task of estimating the the error~$u(t_n)-U_h^n$ of
the MFE approximation becomes more difficult due to the presence of
time discretization errors $e_h^n=u_h(t_n)-U_h^n$, which are added
to the spatial discretization errors $u(t_n)-u_h(t_n)$. However we
show in Section~\ref{sec:4} that if temporal and spatial errors are
not very different in size, the quantity $\widetilde U^n-U_h^n$
correctly esimates the spatial error because the leading terms of
the temporal errors in
 $\widetilde U^n$ and~$U_h^n$ get canceled out when subtracting
$\widetilde U^n-U_h^n$, leaving only the spatial component of the error.
This is a very convenient property that allows to use independent
procedures for the tasks of estimating the errors of the spatial and
temporal discretizations. We remark that the temporal error can be
routinely controlled by resorting to well-known ordinary
differential equations techniques. Analogous results were obtained
in \cite{Javier-Julia-estiman-3} for fully discrete finite element
approximations to evolutionary convection-reaction-diffusion
equations using the backward Euler method.

As in the semidiscrete case, a key point in our results is again the
fact that the fully discrete MFE approximation $(U_h^n, P_h^n)$  to
the Navier-Stokes problem~(\ref{onetwo})--(\ref{ic}) is also the MFE
approximation to the solution~$(\widetilde U^n,\widetilde P^n)$ of
the Stokes problem~~(\ref{posth0n})--(\ref{posth1n}). As a
consequence, we can use again any available error estimator for the
Stokes problem to estimate the spatial error of the fully discrete
MFE approximations $(U_h^n, P_h^n)$ to the Navier-Stokes
problem~(\ref{onetwo})--(\ref{ic}).

Computable mixed finite element approximations
to~$(\widetilde U^n,\widetilde P^n)$, the
so-called fully discrete postprocessed approximations, were studied and analyzed in \cite{jbj_fully} where we proved
that the fully discrete postprocessed approximations maintain the increased spatial accuracy  of the semi-discrete approximations. The analysis in the second
part of the present paper borrows in part from~$\cite{jbj_fully}$.
Also, we propose a computable error estimator based on the fully discrete
postprocessed approximation of \cite{jbj_fully} and show that it also has the excellent property of separating spatial and temporal errors.

The rest of the paper is as follows. In Section 2 we introduce some preliminaries and notation.
In Section 3 we study the
a posteriori error estimation of semi-discrete in space MFE approximations. In Section 4 we study a
posteriori error estimates
for fully discrete approximations. Finally, some numerical experiments are shown in Section 5.
\section{Preliminaries and notations}
\label{sec:2}
We will assume that
$\Omega$ is a bounded domain in~${\mathbb{R}}^{d},\, d=2,3$,
of class~${\cal C}^m$, for $m\ge 2$. When dealing with linear elements
($r=2$ below) $\Omega$ may also be a convex polygonal or polyhedral domain.
We consider
the Hilbert spaces
\begin{align*}
 H&=\left\{ u \in L^{2}(\O)^d \mid\mbox{div}(u)=0, \, u\cdot n_{|_{\partial \Omega}}=0 \right\},\\
V&=\left\{ u \in H^{1}_{0}(\O)^d \mid \mbox{div}(u)=0 \right\},
\end{align*}
 endowed
with the inner product of $L^{2}(\O)^{d}$ and $H^{1}_{0}(\O)^{d}$,
respectively. For  $l\ge 0$ integer and $1\le q\le \infty$, we consider the
standard  spaces, $W^{l,q}(\Omega)^d$, of functions with
derivatives up to order $l$ in $L^q(\Omega)$, and
$H^l(\Omega)^d=W^{l,2}(\Omega)^d$. We will denote by $\|\cdot \|_l$ the norm in $H^l(\Omega)^d$,
and~$\|\cdot\|_{-l}$ will represent
the norm of its dual space. We consider also the quotient spaces
$H^l(\Omega)/{\mathbb R}$ with norm $\| p\|_{H^l/{\mathbb R}}= \inf\{ \| p+c\|_l\mid  c\in {\mathbb R}\}$.

We recall the following Sobolev's imbeddings \cite{Adams}: For
$q \in [1, \infty)$, there exists a constant $C=C(\Omega, q)$ such
that
\begin{equation}\label{sob1}
\|v\|_{L^{q'}} \le C \| v\|_{W^{s,q}}, \,\,\quad
\frac{1}{q'}
\ge \frac{1}{q}-\frac{s}{d}>0,\quad q<\infty, \quad v \in
W^{s,q}(\Omega)^{d}.
\end{equation}
For $q'=\infty$, (\ref{sob1}) holds with $\frac{1}{q}<\frac{s}{d}$.

The following inf-sup condition is satisfied (see \cite{girrav}), there exists a constant $\beta>0$
such that
\begin{equation}\label{inf-sup}
\inf_{q\in L^2(\Omega)/{\Bbb R}}\sup_{ v\in H_0^1(\Omega)^d}\frac{(q,\nabla\cdot v)}{\|v\|_1\|q\|_{L^2/{\Bbb R}}}\ge \beta,
\end{equation}
where, here and in the sequel, $(\cdot,\cdot)$ denotes the standard inner product
in $L^2(\Omega)$ or in $L^2(\Omega)^d$.

Let $\Pi: L^2(\O)^d \lor H$ be the $L^2(\O)^d $ projector onto $H$.
We denote by $A$ the Stokes operator on $\O$: $$ A:
\mathcal{D}(A)\subset H \lor H, \quad \, A=-\Pi\Delta , \quad
\mathcal{D}(A)=H^{2}(\O)^{d} \cap V. $$ Applying Leray's projector
$\Pi$ to (\ref{onetwo}), the equations can be written in the form
\begin{eqnarray*}
u_t +     A u +B(u,u) =\Pi f \quad \mbox{ in } \O,
\end{eqnarray*}
where $B(u,v)=\Pi  (u\cdot \nabla) v$ for $u$, $v$ in $H_0^1(\Omega)^d$.

We shall use the trilinear form $b(\cdot,\cdot,\cdot)$ defined by
$$
\displaylines{
b(u,v,w)=(F(u,v),w)\quad\forall u,v,w\in H_0^1(\Omega)^d,
\cr
\hbox{\rm
where}
\hfill\cr
F(u,v)=(u\cdot \nabla) v +\frac{1}{2}(\nabla\cdot u)v\quad\forall
u,v\in H_0^1(\Omega)^d.
}
$$
It is straightforward to verify that $b$ enjoys skew-symmetry:
\begin{equation}\label{skew}
b(u,v,w)=-b(u,w,v) \quad \forall u,v,w\in H_0^1(\Omega)^d.
\end{equation}
Let us observe that $B(u,v)=\Pi F(u,v)$ for $u\in V,$ $v\in H_0^1(\Omega)^d$.

Let us consider for $\alpha\in {\mathbb R}$ and $t>0$ the operators
$A^{\alpha}$ and~$e^{-tA}$, which are defined
by means of the spectral properties of~$A$
(see, e.g., \cite[p.~33]{Constantin-Foias}, \cite{Fujita-Kato}). Notice that
$A$ is a positive self-adjoint operator with compact resolvent in~$H$.
An easy calculation shows that
\begin{equation}\label{lee1}
\|A^{\alpha}e^{-tA}\|_{0}\leq
(\alpha e^{-1})^\alpha
t^{-\alpha}, \qquad \alpha\ge 0,\ t>0,
\end{equation}
where, here and in what follows, $\left\|\cdot\right\|_0$ when applied to an
operator denotes the associated operator norm.

We shall assume that the solution $u$ of (\ref{onetwo})-(\ref{ic})
satisfies
\begin{equation}
\label{eq:M_1}
\|u(t)\|_1\le M_1,\quad
\|u(t)\|_2\le M_2,\quad 0\le t\le T,
\end{equation}
for some constants $M_1$ and~$M_2$. We shall also assume that there exists
a constant $\tilde M_2$ such that
\begin{equation}
\|f\|_1+\|f_t\|_1
+\|f_{tt}\|_1\le \tilde M_2,\quad 0\le t\le T.
\label{tildeM2}
\end{equation}
Finally, we shall assume that for some $k\ge 2$
$$
\sup_{0\le t\le T}
\bigl\| \partial_t^{\lfloor k/2\rfloor} f\bigr\|_{k-1-2{\lfloor k/2\rfloor}}+
\sum_{j=0}^{\lfloor (k-2)/2\rfloor}\sup_{0\le t\le T}
\bigl\| \partial_t^j f\bigr\|_{k-2j-2}
<+\infty,
$$
so that, according to Theorems~2.4 and~2.5 in~\cite{heyran0},
there exist positive constants $M_k$ and $K_{k}$ such that the
following bounds hold:
\begin{eqnarray}
\| u(t)\|_{k} + \| u_t(t)\|_{k-2}
+\| p(t)\|_{H^{k-1}/{\mathbb R}} \le M_k \tau(t)^{1-k/2},
\qquad\qquad\quad
\label{eq:u-inf}
\\
\quad
\int_0^t
\sigma_{k-3}(s)
\bigl(\| u(s)\|_{k}^2 + \| u_s(s)\|_{k-2}^2
+\| p(s)\|_{H^{k-1}/{\mathbb R}}^2+
\| p_s(s)\|_{H^{k-3}/{\mathbb R}} ^2\bigl)\,{\rm d}s\le
K_{k}^2,
\label{eq:u-int}
\end{eqnarray}
where $\tau(t)=\min(t,1)$ and~$\sigma_n=e^{-\alpha(t-s)}\tau^n(s)$ for some
$\alpha>0$. Observe that for $t\le T<\infty$, we can take~$\tau(t)=t$ and
$\sigma_n(s)=s^n$. For simplicity, we will take these values of
$\tau$ and~$\sigma_n$.

Let $\mathcal{T}_{h}=(\tau_i^h,\phi_{i}^{h})_{i \in
I_{h}}$, $h>0$ be a family of partitions of suitable domains $\Omega_h$, where
$h$ is the maximum diameter of the elements $\tau_i^h\in \mathcal{T}_{h}$,
and $\phi_i^h$ are the mappings of the reference simplex
$\tau_0$ onto $\tau_i^h$.

Let $r \geq 2$, we consider the finite-element spaces
$$
S_{h,r}=\left\{ \chi_{h} \in \mathcal{C}\left(\overline{\O}_{h}\right) \,  |
\, {\chi_{h}}{|_{\tau_{i}^{h}}}
\circ \phi^{h}_{i} \, \in \, P^{r-1}(\tau_{0})  \right\} \subset H^{1}(\O_{h}),
\
{S}_{h,r}^0= S_{h,r}\cap H^{1}_{0}(\O_{h}),
$$
where $P^{r-1}(\tau_{0})$ denotes the space of polynomials
of degree at most $r-1$ on $\tau_{0}$. As it is customary in the analysis
of finite-element methods for the Navier-Stokes equations
(see e.~g., \cite{Bause}, \cite{heyran0},
\cite{heyran2}, \cite{HR-IV}, \cite{Hill-Suli})  we restrict ourselves to quasiuniform and regular meshes $\mathcal{T}_{h}$, so that as a consequence
of   \cite[Theorem 3.2.6]{ciar0}, the
following inverse inequality holds
for each $v_{h} \in ({S}_{h,r}^0)^{d}$
\begin{eqnarray}
\label{inv}
\| v_{h} \|_{W^{m,q}(\Omega_h)^{d}} \leq C h^{l-m-d\left(\frac{1}{q'}-\frac{1}{q}\right)}
\|v_{h}\|_{W^{l,q'}(\Omega_h)^{d}},
\end{eqnarray}
where $0\leq l \leq m \leq 1$, $1\leq q' \leq q \leq \infty$.

We shall denote by $(X_{h,r}, Q_{h,r-1})$
the so-called Hood--Taylor element \cite{BF,hood0}, when $r\ge 3$, where
$$
X_{h,r}=\left({S}_{h,r}^0\right)^{d},\quad
Q_{h,r-1}=S_{h,r-1}\cap L^2(\O_{h})/{\mathbb R},\quad r
\ge 3,
$$
and the so-called mini-element~\cite{Brezzi-Fortin91} when $r=2$,
where $Q_{h,1}=S_{h,2}\cap L^2(\O_{h})/{\mathbb R}$, and
$X_{h,2}=({S}_{h,2}^0)^{d}\oplus{\mathbb B}_h$. Here,
${\mathbb B}_h$ is spanned by the bubble functions $b_\tau$,
$\tau\in\mathcal{T}_h$, defined by
$b_\tau(x)=(d+1)^{d+1}\lambda_1(x)\cdots
 \lambda_{d+1}(x)$,  if~$x\in \tau$ and 0 elsewhere,
where  $\lambda_1(x),\ldots,\lambda_{d+1}(x)$ denote the
 barycentric coordinates of~$x$. For these elements a uniform inf-sup condition is satisfied
(see \cite{BF}), that is,
there exists a constant $\beta>0$ independent of the mesh grid size $h$ such that
\begin{equation}\label{lbbh}
 \inf_{q_{h}\in Q_{h,r-1}}\sup_{v_{h}\in X_{h,r}}
\frac{(q_{h},\nabla \cdot v_{h})}{\|v_{h}\|_{1}
\|q_{h}\|_{L^2/{\mathbb R}}} \geq \beta.
\end{equation}
We remark that our analysis can also be applied to other pairs of LBB-stable mixed finite elements (see \cite[Remark 2.1]{jbj_regularity}).

The approximate velocity belongs to the discrete
divergence-free space
$$
V_{h,r}=X_{h,r}\cap \left\{ \chi_{h} \in H^{1}_{0}(\O_{h})^d \mid
(q_{h}, \nabla\cdot\chi_{h}) =0  \quad\forall q_{h} \in Q_{h,r-1}
\right\},
$$
which is not a subspace of $V$. We shall frequently write~$V_h$ instead
of~$V_{h,r}$ whenever the value of~$r$ plays no particular role.

Let $\Pi_{h}:L^{2}(\O )^{d}\lor V_{h,r}$ be the discrete Leray's
projection defined by
$$(\Pi_{h}u,\chi_{h} )=(u,\chi_{h})\quad \forall\chi_{h}\in V_{h,r}.$$
We will use the following
well-known bounds
\begin{equation}
\|(I-\Pi_{h})u\|_j \le Ch^{l-j}\|u\|_l,\quad 1\le l\le 2,\quad j=0,1. \label{eq:error-Pi_h}
\end{equation}
We will denote by $A_h:V_{h}\rightarrow V_{h}$ the discrete Stokes operator defined by
$$
(\nabla v_h,\nabla\phi_h)=(A_h v_h,\phi_h)=\left(A_h^{1/2}v_h,
A_h^{1/2}\phi_h\right) \quad\forall v_h,\phi_h\in V_{h}.
$$
Let
$(u,p)\in  (H^2(\Omega)^d\cap V )\times ( H^1(\Omega)\slash \mathbb{R} )$
be the solution of a Stokes problem with right-hand side $g$,
we will denote by $s_h=S_{h}(u)\in V_{h}$ the so-called Stokes projection
(see \cite{heyran2}) defined as the velocity component of solution
of the following Stokes problem:
find $(s_h,q_h)\in(X_{h,r},Q_{h,r-1})$ such that
\begin{align}
   (\nabla s_h,\nabla\phi_h)+(\nabla q_h,\phi_h)&=(g,\phi_h)
&&\forall \phi_h\in X_{h,r},&&&&
\label{stokesnew}\\
(\nabla \cdot s_h,\psi_h)&=0&& \forall \psi_h\in  Q_{h,r-1}.&&&&
\label{stokesnew2}
\end{align}
The following bound holds for $2\le l\le r$:
\begin{equation}
\|u-s_h\|_0+h\|u-s_h\|_1\le C h^l \bigl (\|u\|_l+\|p\|_{H^{l-1}/{\mathbb R}}\bigr ).
\label{stokespro}
\end{equation}
The proof of (\ref{stokespro}) for $\Omega=\Omega_h$ can be found
in~\cite{heyran2}. For the general case, $\Omega_h$ must be such that
the value of~$\delta(h)=\max_{x\in \partial\Omega_h}\hbox{\rm dist}
(x,\partial\Omega)$ satisfies
$
\delta(h)=O(h^{2(r-1)})$.
This can be achieved if, for example, $\partial \Omega$ is piecewise
of class~${\cal C}^{2(r-1)}$, and super\-para\-metric approximation
at the boundary is used \cite{bb}. Under the same conditions, the
bound for the pressure is \cite{girrav}
\begin{equation}\label{stokespre}
\|p-q_h\|_{L^2/{\mathbb R}}\le
C_\beta h^{l-1}\bigl (\|u\|_l+\|p\|_{H^{l-1}/{\mathbb R}}\bigr ),
\end{equation}
where the constant $C_\beta$ depends on the constant $\beta$ in the inf-sup condition
(\ref{lbbh}).
We will assume that the domain~$\Omega$ is of class ${\cal C}^m$, with $m\ge r$ so that
standard bounds for the Stokes problem \cite{bb}, \cite{Galdi} imply that

\begin{equation}
\label{eq:may06-nueva}
\bigl\| A^{-1} \Pi g\bigr\|_{2+j} \le \left\| g\right\|_j,\qquad
-1\le j\le m-2.
\end{equation}
For a domain~$\Omega$ of class ${\cal C}^2$ we also have the bound (see \cite{cata})
\begin{equation}\label{preHR}
\|p\|_{H^1/{\Bbb R}}\le c \|g\|_0.
\end{equation}
%
%Then, using
%standard duality arguments and~(\ref{stokespro}), it is easy to show that
%
%\begin{equation}\label{die7r2p}
%\| u-s_{h}\|_{-m} \le C h^{l+\min{(m,r-2)}} (\|u\|_{l} +
%\|p\|_{H^{l-1}/{\mathbb R}}), \quad m=1,2.
%\end{equation}
In what follows we will apply the above estimates to
the particular case in which $(u,p)$ is the solution of the
Navier--Stokes problem (\ref{onetwo})--(\ref{ic}). In that case $s_h=S_h(u)$ is the discrete
velocity in problem (\ref{stokesnew})--(\ref{stokesnew2}) with $g=f-u_t-(u\cdot \nabla u)$.
Note that the temporal variable $t$ appears here merely as a parameter, and then,
taking the time derivative, the error bound
(\ref{stokespro})
can also be applied to the time derivative of $s_h$ changing
$u$, $p$ by $u_t$, $p_t$.

Since we are assuming that~$\Omega$ is of class~${\cal C}^m$ and $m\ge 2$,
from (\ref{stokespro}) and standard bounds for the Stokes
problem~\cite{bb,Galdi},
we deduce that
\begin{equation}
\left\|\left(A^{-1}\Pi -A_h^{-1}\Pi_h\right) f\right\|_j\le  Ch^{2-j}\|f\|_0\quad
\forall f\in L^2(\Omega)^d,
\quad j=0,1.
\label{stokespro+1}
\end{equation}
We consider the semi-discrete finite-element approximation~$(u_h,p_h)$ to~$(u,p)$,
solution
of~(\ref{onetwo})--(\ref{ic}).
That is,
given $u_h(0)=\Pi_hu_{0}$, we compute
$u_{h}(t)\in X_{h,r}$ and $p_{h}(t)\in Q_{h,r-1}$, $t\in(0,T]$,   satisfying
\begin{align}\label{ten}
(\dot u_{h}, \phi_{h})
+     ( \nabla u_{h}, \nabla \phi_{h}) + b(u_{h}, u_{h}, \phi_{h}) +
( \nabla p_{h}, \phi_{h}) & =  (f, \phi_{h})&& \forall \, \phi_{h} \in X_{h,r},\\
(\nabla \cdot u_{h}, \psi_{h}) & = 0&&\forall \, \psi_{h} \in Q_{h,r-1}.
\label{ten2}
\end{align}
%For convenience, we rewrite this problem in the following way,
%\begin{equation}\label{atenop}
% \dot{u}_{h}+     A_{h} u_{h}+ B_{h}(u_{h}, u_{h})= \Pi_{h}f,\quad
% u_h(0)=\Pi_{h}u_0,
% \end{equation}
%where $B_h(u,v)=\Pi_{h}F(u,v)$.

For $2\le r\le 5$, provided that~(\ref{stokespro})--(\ref{stokespre}) hold for
$l\le r$,
and~(\ref{eq:u-inf})--(\ref{eq:u-int}) hold for $k=r$, then we have
\begin{equation}
\label{eq:err_vel(t)}
\| u(t)-u_h(t)\|_0+h\| u(t)-u_h(t)\|_1
\le C\frac{h^r}{t^{(r-2)/2}},
\quad 0\le t\le T,
\end{equation}
(see, e.g.,
\cite{jbj_regularity,heyran0,heyran2}),
and also,
\begin{equation}
\label{eq:err_pre(t)}
\| p(t)-p_h(t)\|_{L^2/{\mathbb R}} \le C\frac{h^{r-1}}{t^{(r'-2)/2}},
\quad 0\le t\le T,
\end{equation}
where $r'=r$ if $r\le 4$ and $r'=r+1$ if~$r=5$.

see \cite[Proposition~3.2]{HR-IV}.
\section{A posteriori error estimations. Semidiscrete case}
\label{sec:3} Let us consider the MFE approximation $(u_h,p_h)$ at
any time ${t^*}\in(0,T]$ to $(u({t^*}),p({t^*}))$ obtained by
solving (\ref{ten})--(\ref{ten2}). We consider the postprocessed
approximation $(\tilde u(t^*),\tilde p(t^*))$ in
$(V,L^2(\Omega)/{\Bbb R})$ which is the solution of the following
Stokes problem written in weak form
\begin{eqnarray}\label{posth0}\qquad
 \big(\nabla\tilde u(t^*),
 \nabla {\phi}\big) \,{+}\,\big(\nabla \tilde p(t^*),
 {\phi} \big)
&=& (f, {\phi}) \,{-}\, b(u_{h}({t^*}),u_{h}({t^*}),{\phi}) \,{-}\,
 ( \dot u_{h}({t^*}), {\phi}),\quad
 \\
\bet \hspace*{15pt}\big( \nabla \cdot \tilde u(t^*), {\psi} \big) &=
& 0, \label{posth1}
\end{eqnarray}
for all ${\phi} \in H_0^1(\Omega)^d$ and ${\psi} \in
L^2(\Omega)/{\Bbb R}$. We remark that the MFE approximation
$(u_h({t^*}),p_h({t^*}))$ to $(u({t^*}),p({t^*}))$ is also the MFE
approximation to the solution $(\tilde u(t^*),\tilde p(t^*))$ of
 the Stokes
problem (\ref{posth0})--(\ref{posth1}). In Theorems~\ref{semi_disve}
and~\ref{semi_dispre} below we prove that the postprocessed
approximation $(\tilde u(t^*),\tilde p(t^*))$ is an improved
approximation to the solution $(u,p)$ of the evolutionary
Navier-Stokes equations (\ref{onetwo})--(\ref{ic}) at time ${t^*}$.
Although, as it is obvious, $(\tilde u(t^*),\tilde p(t^*))$ is not
computable in practice, it is however a useful tool to provide a
posteriori error estimates for the MFE approximation $(u_h,p_h)$ at
any desired time ${t^*}>0$. In Theorem~\ref{semi_disve} we obtain
the error bounds for the velocity  and in Theorem~\ref{semi_dispre}
the bounds for the pressure. The improvement is achieved in the
$H^1(\Omega)^d$ norm when using the mini-element ($r=2$) and in both
the $L^2(\Omega)^d$ and $H^1(\Omega)^d$ norms in the cases $r=3,4$.
%The proof of Theorem~\ref{semi_disve} requires some previous results
%which we now state and prove.

In the sequel we will use that for a forcing term satisfying
(\ref{tildeM2}) there exists a constant $\tilde M_3>0$, depending
only on~$\tilde M_2$, $\| A_hu_h(0)\|_0$ and $\sup_{0\le t\le T} \|
u_h(t)\|_1$, such that the following bound hold for $0\le t\le T$:
\begin{equation}\label{A_h_uh}
\| A_h u_h(t)\|_0^2\le \tilde{M}_3^2,
\end{equation}
The following inequalities hold for all $v_h,w_h\in V_{h}$
and~$\phi\in H^1_0(\Omega)^d$, see~\cite[(3.7)]{HR-IV}:
\begin{eqnarray}
\label{eq:adelanto0} |b(v_h,v_h,\phi)|& \le &
c\|v_h\|_1^{3/2}\|A_hv_h\|_0^{1/2} \|\phi\|_0,
\\
\label{eq:adelanto1} |b(v_h,w_h,\phi)|+|b(w_h,v_h,\phi)| &\le&
c\|v_h\|_1 \|A_hw_h\|_0 \|\phi\|_0.
\end{eqnarray}

The proof of Theorem~\ref{semi_disve} requires some previous results
which we now state and prove.

We will use the fact that $\bigl\| A_h^{1/2} w_h\bigr\|_0= \left\|
\nabla w_h\right\|_0$ for $w_h\in V_h$, from where it follows that
\begin{equation}
\label{A-1equiv}
C^{-1} \bigl\| A_h^{-1/2} w_h\bigr\|_0 \le  \left\| w_h\right\|_{-1}
\le  C\bigl\| A_h^{-1/2} w_h\bigr\|_0 \quad \forall w_h \in V_h,
\end{equation}
where the constant $C$ is independent of~$h$.

%The bounds of next lemma are taken from~\cite[Lemma~3.1]{bbj} but requiring now a weaker threshold condition.

%Also
%we will use the following bound from~\cite[Lemma~4.4]{jbj_regularity}, which
%is valid for $w_h^1,w_h^2\in V_h$ satisfying the threshold condition
%$\|w_h^l - u\|_j \le \alpha h^{3/2-j}$ for $j=0,1$ and for some $\alpha>0$. Since
%in~\cite[Lemma~4.4]{jbj_regularity} the threshold condition used in the proof is
%stronger than the one we require now we include the proof of the following lemma.

\begin{lemma} Let $(u,p)$ be the solution of (\ref{onetwo})--(\ref{ic}) and
fix $\alpha>0$. Then there exists a positive constant $C=C(M_2,\alpha)$ such
that for $w_h^1,w_h^2\in V_h$
satisfying the threshold condition
\begin{equation}\label{threshold}
\|w^l_h - u\|_j \le \alpha h^{3/2-j},\quad j=0,1,\quad l=1,2,
\end{equation}
the following inequalities hold for $j=0,1$:
\begin{eqnarray}
\label{4.21-4.22}
\,\,\,\bigl\| A_h^{-j/2}\Pi_h (F(w_h^1,w_h^1)-F(w_h^2,w_h^2))\bigr\|_0&\le &C
\bigl\| A_h^{(1-j)/2}(w_h^1-w_h^2)\bigr\|_0,\\
\label{4.21-4.22b}
\bigl\| A_h^{-j/2}\Pi_h (F(w_h^1,w_h^1)-F(u,u))\bigr\|_0&\le& C
\bigl\| w_h^1-u\bigr\|_{1-j}.
\end{eqnarray}
\end{lemma}
\begin{proof}
%We prove~(\ref{4.21-4.22}) since the proof of~(\ref{4.21-4.22b}) is similar but yet simpler.
Due to the equivalence~(\ref{A-1equiv}) and
%that as a consequence
%of~(\ref{eq:error-Pi_h}) and the stability in~$H^1$ of the Leray operator
%we have $\left\|\Pi_hf\right\|_1\le \left\|f\right\|_1$ for $f\in H^1(\Omega)$,
and since $\|\Pi_h f\|_0\le \|f\|_0$ for $f\in L^2(\Omega)^d$ it is sufficient to prove
\begin{equation}\label{eq:aux_bat}
\| F(w_h^1,w_h^1) -F(w,w)\|_{-j} \le
  C  \|w_h^1-w\|_{1-j},\quad j=0,1,
\end{equation}
for $w=w_h^2$ or $w=u$.
We follow the proof \cite[Lemma~3.1]{bbj} where a different threshold assumption is assumed.
We do this for $w=w_h^2$, since the
case $w=u$ is similar but yet simpler. We write
\begin{equation}\label{eq:aux0}
F(w_h^1,w_h^1)-F(w_h^2,w_h^2)  =  F(w_{h}^1,e_{h})+F(e_{h},w_{h}^2),
\end{equation}
where $e_h=w_h^1-w_h^2$.
We first observe that
\begin{eqnarray*}
\|F(e _{h},w_{h}^2)\|_{0}& =& \sup_{\|\phi\|_{0}=1}\bigg|
(e_{h}\cdot \nabla w_{h}^2),\phi)+ {1\over 2}((\nabla \cdot e_{h})w_{h}^2),\phi)\bigg|\\
&\le & C\|e_{h}\|_{L^{2d}}\|\nabla w_{h}^2\|_{L^{2d/(d-1)}}
+C \|e_{h}\|_{1}\|w_{h}^2\|_{L^\infty}\\
&\le & C\bigl(\|\nabla w_{h}^2\|_{L^{2d/(d-1)}}+\|w_{h}^2\|_{L^\infty}
\bigr) \|e_{h}\|_{1},
\end{eqnarray*}
where, in the last inequality, we have used that thanks to Sobolev's inequality
(\ref{sob1}) we have $\|e_{h}\|_{L^{2d}}\le C\|e_{h}\|_{1}$.
Similarly,
\begin{eqnarray*}
\|F(w_{h}^1,e_{h})\|_{0}& \le &
C\|w_{h}^1\|_{L^\infty} \|e_{h}\|_{1} +C\|\nabla w_{h}^1\|_{L^{2d/(d-1)}}
\|e_{h}\|_{L^{2d}}\\
&\le &
C\bigl(\|w_{h}^1\|_{L^\infty}+\|\nabla w_{h}^1\|_{L^{2d/(d-1)}}\bigr)
 \|e_{h}\|_{1}.
\end{eqnarray*}
The proof of the case $j=0$ in~(\ref{eq:aux_bat}) is finished if we show
that for $l=1,2$, both~$\|w_{h}^l\|_{L^\infty}$ and~$\|\nabla w_{h}^l\|_{L^{2d/(d-1)}}$
are bounded in terms of~$M_2$ and the value~$\alpha$ in the threshold
assumption~(\ref{threshold}). To do this, we will use
the inverse inequality~(\ref{inv}) and the fact that the Stokes projection $s_h=S_h(u)$
satisfies that
$$
\|s_{h}\|_{L^\infty} \le C_s,\qquad \|\nabla s_{h}\|_{L^{2d}}\le C_s
$$
for some constant $C_s=C_s(M_2)$ (see for example the proof of~Lemma~3.1
in~\cite{bbj}).
We have
$$
\|w_{h}^l\|_{L^\infty}\le \|w_{h}^l-s_h\|_{L^\infty}+\|s_{h}\|_{L^\infty}
\le Ch^{-d/2} \|w_{h}^l-s_h\|_{0}+\|s_{h}\|_{L^\infty},
$$
where in the last inequality we have applied~(\ref{inv}), and, similarly,
\begin{eqnarray*}
\|\nabla w_{h}^l\|_{L^{2d/(d-1)}}&\le &
\|\nabla (w_{h}^l-s_h)\|_{L^{2d/(d-1)}}+
\|\nabla s_{h}\|_{L^{2d/(d-1)}}\\
&\le&
Ch^{-1/2} \|\nabla (w_{h}^l-s_h)\|_{0}+\|\nabla s_{h}\|_{L^{2d}},
\end{eqnarray*}
where we also have used that $\|\cdot \|_{L^{p}}
\le \|\cdot\|_{L^{p'}}$ for $p<p'$.
Now the threshold assumption~(\ref{threshold}) and~(\ref{stokespro}) show
the boundedness of
$\|w_{h}^l\|_{L^\infty}$ and~$\|\nabla w_{h}^l\|_{L^{2d/(d-1)}}$.

Finally, the proof of the case~$j=1$ in~(\ref{eq:aux_bat}) is, with obvious changes,
that of the equivalent result in~\cite[Lemma 3.1]{bbj}.
\end{proof}

In the sequel we consider the auxiliary function $v_h:[0,T]\rightarrow V_h$ solution of
\begin{equation}\label{atenopv}
% \dot{v}_{h}+\nu A_{h} v_{h}+ \Pi_{h}F(u, u)= \Pi_{h}f,
 \dot{v}_{h}+     A_{h} v_{h}+ \Pi_h F(u, u)= \Pi_{h}f,\qquad
 v_h(0)=\Pi_{h}u_0.
 \end{equation}
According to~\cite[Remark~4.2]{jbj_regularity} we have
\begin{equation}\label{eq:cota-z-low}
\max_{0\le t\le T}\|v_h(t)-\Pi_hu(t)\|_0\le C |\log(h)|h^{2},\
\end{equation}
for some constant $C=C(M_2)$. The following lemma provides a
superconvergence result.

\begin{lemma}\label{le:sup1}
Let $(u,p)$ be the solution of (\ref{onetwo})--(\ref{ic}). Then, there exists a positive constant $C$ such that the solution $v_h$ of~(\ref{atenopv}) and the Galerkin
approximation~$u_h$ satisfy the following bound,
\begin{equation}\label{supercuad}
 \|v_{h}(t)-u_{h}(t)\|_{1} \le
 {C}|\log(h)|^2
 h^{2},\quad t\in(0,T].
\end{equation}
\end{lemma}
\begin{proof} Since for $y_h=A_h^{1/2}(v_h-u_h)$ we have
$$ \dot y_h +A_h y_h
+A_h^{1/2}\Pi_h(F(v_h,v_h)-F(u_h,u_h)) =A_h^{1/2}\rho_h ,
$$
where $\rho_h=\Pi_h(F(v_h,v_h)-F(u,u))$, it follows that
\begin{align*}
\left\| y_h(t)\right\|_0 \le &
\int_0^t \bigl\|A_h^{1/2}e^{-(t-s)A_h}\bigr\|_0
\left\|\Pi_h(F(v_h,v_h)-F(u_h,u_h))\right\|_0
\\
&{}+
\int_0^t \bigl\|A_h e^{-(t-s)A_h}
(A_h^{-1/2}\rho_h(s))\bigr\|_0
\,ds.
\end{align*}
Applying (\ref{4.21-4.22}) we have
$\left\|\Pi_h(F(v_h,v_h)-F(u_h,u_h))\right\|_0\le
C\left\|y_h\right\|_0$, so that
taking into account that
\begin{eqnarray}\label{eq:vabien}
\bigl\|A_h^{1/2}e^{-(t-s)A_h}\bigr\|_0
\le (2e(t-s))^{-1/2},
\end{eqnarray}
 it follows that
$$
\left\| y_h(t)\right\|_0 \le
\frac{1}{\sqrt{2e}}\int_0^t \frac{\left\|y_h(s)\right\|_0}{\sqrt{t-s}}
+
\int_0^t \bigl\|A_h e^{-(t-s)A_h}
(A_h^{-1/2}\rho_h(s))\bigr\|_0
\,ds.
$$
Since applying~\cite[Lemma~4.2]{jbj_regularity} we obtain
$$
\int_0^t \bigl\|A_h e^{-(t-s)A_h}
(A_h^{-1/2}\rho_h(s))\bigr\|_0
\,ds\le C |\log(h)|\max_{0\le s\le t}\|\rho_h(s)\|_0,
$$
a generalized Gronwall lemma \cite[pp.~188-189]{Henry}, together with
(\ref{4.21-4.22})
allow us to conclude
$$
\|v_h-u_h\|_1\le C|\log(h)| \|v_h-u\|_0.
$$
Then by writing $\|v_h-u\|_0\le \|v_h-\Pi_h u\|_0+\|\Pi_h u-u\|_0$
and applying~(\ref{eq:error-Pi_h}) and~(\ref{eq:cota-z-low}), the proof is finished
if we check that the threshold condition~(\ref{threshold}) holds
for $w_h^1=u_h$ and~$w_h^2=v_h$.
In view of~(\ref{eq:cota-z-low}), (\ref{eq:error-Pi_h}) and the
inverse inequality~(\ref{inv}) we have indeed that $\left\| v_h-u\right\|_j=o(h^{3/2-j})$,
for $j=0,1$.
In the case of~$u_h$ the threshold condition holds due to (\ref{eq:err_vel(t)}).
%since $u_h(0)=\Pi_hu_0$ we have that $\left\| u_h(0)-u(0)\right\|_j\le
%Ch^{2-j}=o(h^{3/2-j})$,
%so that a standard continuity argument and~(\ref{supercuad}) easily show that
%$\left\| u_h(t)-u(t)\right\|_j\le =o(h^{3/2-j})$, for $t\in(0,T]$.
\end{proof}

\begin{lemma}
\label{le:e_t}
Let $(u,p)$ be the solution of {\rm
(\ref{onetwo})--(\ref{ic})}. Then there exists a positive
constant $C$ such that
\begin{equation}
\|\dot v_h(t)-\dot u_h(t)\|_{-1} \le  {C}|\log(h)|^2
h^{2}, \quad t\in(0,T],
\label{e_t-1}
\end{equation}
where $v_h$ and $u_{h}$ are defined by
 {\rm (\ref{atenopv}) }and {\rm (\ref{ten})-(\ref{ten2})} respectively.
\end{lemma}

\begin{proof} The difference $v_h-u_h$ satisfies that
$\dot v_h-\dot u_h=A_h(v_h-u_h)+\Pi_h(B(u,u) -B(u_h,u_h)$, so that
multiplying by
$A_h^{-1/2}$ and taking norms,
thanks to~(\ref{4.21-4.22b}), we have
$$
\bigl\| A_h^{-1/2} (\dot v_h-\dot u_h)\bigr\|_0 \le \bigl\| A_h^{1/2}(v_h-u_h)\bigr\|_0+
C\left\|u-u_h\right\|_0.
$$
Now we write
$$
\left\|u-u_h\right\|_0\le \left\|u-\Pi_hu\right\|_0+\left\|\Pi_h u-v_h\right\|_0+
\left\|v_h-u_h\right\|_0,
$$
so that~(\ref{eq:error-Pi_h}), (\ref{eq:cota-z-low}) and~(\ref{supercuad}),
%together with the fact that $\bigl\| A_h^{1/2} w_h\bigr\|_0=
%\left\| \nabla w_h\right\|_0$ for $w_h\in V_h$,
allow us to write,
$$
\bigl\| A_h^{-1/2} (\dot v_h-\dot u_h)\bigr\|_0  \le C\left|\log(h)\right|^2h^2.
$$
Then, applying (\ref{A-1equiv}) the proof is finished.
\end{proof}

\begin{lemma}\label{le:z_t}
Let $(u,p)$ be the solution of {\rm
(\ref{onetwo})--(\ref{ic})}. Then there exists a positive
constant $C$ such that
\begin{equation}
\|u_t-\dot u_h(t)\|_{-1} \le  \frac{C}{t^{(r-1)/2}}
h^{r}\left|\log(h)\right|^{r'}, \quad t\in(0,T], \quad r=2,3,4,
\end{equation}
where $r'=2$ when $r=2$ and $r'=1$ otherwise.
\end{lemma}

\begin{proof} The case $r=3,4$ is proved in \cite[Lemma~5.1]{jbj_regularity}.
For the case $r=2$ we write
\begin{equation}\label{demini1}
u_t-\dot u_h=(u_t-\Pi_h u_t)+(\Pi_h u_t-\dot v_h)+(\dot v_h-\dot
u_h).
\end{equation}
A simple duality argument and the fact that  $\|u_t-\Pi_h
u_t\|_0\le Ch \|u_t\|_1$, easily show that
$$
\|(I-\Pi_h)u_t\|_{-1}\le C h^2\|u_t\|_1\le C\frac{M_3}{t^{1/2}} h^2.
$$
The bound of the third term on the right-hand side of~(\ref{demini1}) is
given in Lemma~\ref{le:e_t}, so that, thanks to the equivalence~(\ref{A-1equiv})
we are left with estimating
$$y_h=t^{1/2}A_h^{-1/2}(\Pi_h u_t-\dot v_h).$$
We notice that
$$
 \dot y_h+A_h y_h=t^{1/2} A_h^{1/2}\dot
 \theta_h+\frac{1}{2}t^{-1/2}A _h^{-1/2}(\Pi_h u_t-\dot v_h),
 $$
where $\theta_h=(\Pi_h-S_h)u$. Thus,
\begin{eqnarray*}
y_h(t)&=&\int_0^t s^{-1/2}A_h^{1/2}e^{-(t-s)A_h} \bigl(s\dot\theta_h\bigr)\,ds\nonumber\\
&&\quad+\frac{1}{2}
\int_0^t s^{-1/2}A_h^{1/2}e^{-(t-s)A_h}A_h^{-1}(\Pi_h u_s-\dot v_h)\,ds.
\end{eqnarray*}
Recalling (\ref{eq:vabien}) by means of
the change of
variables $\tau=s/t$ it is easy to show that
\begin{equation}
\int_0^ts^{-1/2}\bigl\| A_h^{1/2}e^{-(t-s)A_h}\bigr\|_0\,ds
\le \frac{1}{\sqrt{2e}} B(\frac{1}{2},\frac{1}{2}),
\label{labeta}
\end{equation}
where $B$ is the Beta function (see e. g., \cite{Abramowitz-Stegun}).
Thus, we have
$$
 \|y_h\|_0\le CB(\frac{1}{2},\frac{1}{2})\max_{0\le s\le t}\left(s\|\dot
 \theta_h\|_0+\|A_h^{-1}(\Pi_h u_s-\dot v_h)\|_0\right).
 $$
The first term on the right-hand side above is bounded by $CM_4 h^2$. For the
second one we notice hat
$$
A_h^{-1}(\Pi_h u_t-\dot v_h)=\theta_h-(\Pi_h u- v_h)
$$
so that using (\ref{eq:error-Pi_h}), (\ref{stokespro}) and (\ref{eq:cota-z-low}) it is bounded by $M_2h^2\left|\log(h)\right|$.
\end{proof}

\begin{theorem}\label{semi_disve}
Let $(u,p)$ be the solution of (\ref{onetwo})-(\ref{ic}). Then, there exists a positive constant $C$ such that the postprocessed velocity
 $\tilde u$, defined in {\rm (\ref{posth0})-(\ref{posth1})}, satisfies the following bounds:

 {(i)}  If $r=2$ then
\begin{eqnarray}\label{utilde_lineal}
\|u({t^*})-\tilde u({t^*})\|_1\le \frac{C}{{t^*}^{(1/2)}} h^2|\log(h)|^2.
\end{eqnarray}

(ii) If $r=3,4$ then
\begin{eqnarray}\label{utilde_mas}
\|u({t^*})-\tilde u({t^*})\|_j\le \frac{C}{{t^*}^{(r-1)/2}}{h^{r+1-j}}|\log(h)|,\quad j=0,1.
\end{eqnarray}

\end{theorem}
\begin{proof}
The proof follows the same steps as \cite[Theorem~5.2]{jbj_regularity}.
Subtracting~(\ref{posth0}) from~(\ref{onetwo}), standard duality arguments show that
$$
\|\tilde u({t^*})-u({t^*})\|_1\le C\bigl(\|F(u({t^*}),u({t^*}))-F(u_h({t^*}),u_h({t^*}))\|_{-1}+\|u_t({t^*})-\dot u_h({t^*})\|_{-1}\bigr).
$$
To bound the second term on the right-hand side above we apply Lemma~\ref{le:z_t},
whereas for the second we apply (\ref{eq:aux_bat}) to get
\begin{equation}\label{cota_fuuh}
\|F(u({t^*}),u({t^*}))-F(u_h({t^*}),u_h({t^*}))\|_{-1}\le C\|u({t^*})-u_h({t^*})\|_0.
\end{equation}
so that applying (\ref{eq:err_vel(t)}) the proof
of~(\ref{utilde_lineal}) and the case $j=1$ of~(\ref{utilde_mas})
are finished.

We now get the error bounds in the $L^2$ norm.
It is easy to see that
$$
A(\tilde u({t^*})-u({t^*}))=\Pi (F(u({t^*}),u({t^*}))-F(u_h({t^*}),u_h({t^*})))+\Pi(u_t({t^*})-\dot u_h({t^*})).
$$
Then, by applying $A^{-1}$ to both sides of the above equations, we obtain
\begin{align*}
\|\tilde u({t^*})-u({t^*})\|_0 \le &\|A^{-1}\Pi(F(u({t^*}),u({t^*}))-F(u_h({t^*}),u_h({t^*})))\|_0\\
&{}+\|A^{-1}\Pi (u_t({t^*})-\dot u_h({t^*}))\|_0.
\end{align*}

As regards the nonlinear term, applying \cite[Lemma~4.1]{jbj_regularity} we obtain
\begin{multline*}
\|A^{-1}\Pi(F(u({t^*}),u({t^*}))-F(u_h({t^*}),u_h({t^*})))\|_0\le\\
C(\|u({t^*})-u_h({t^*})\|_{-1}+\|u({t^*})-u_h({t^*})\|_1\|u({t^*})-u_h({t^*})\|_0).
\end{multline*}
To bound the second term on the right-hand side above we apply (\ref{eq:err_vel(t)}), whereas the first one is bounded in the proof of \cite[Theorem~5.2]{jbj_regularity} by
$$
\|u({t^*})-u_h({t^*})\|_{-1}\le \frac{C}{{t^*}^{(r-2)/2}}{h^{r+1}}|\log(h)|.
$$
Finally, to bound  $\|A^{-1}\Pi (u_t({t^*})-\dot u_h({t^*}))\|_0$
we apply \cite[Lemma~5.1]{jbj_regularity} to obtain
$$
\|A^{-1}\Pi (u_t({t^*})-\dot u_h({t^*}))\|_0\le \frac{C}{{t^*}^{(r-1)/2}}h^{r+1}|\log(h)|,
$$
which concludes the proof.
\end{proof}

In the following theorem we obtain the error bounds for the pressure $\tilde p$.
\begin{theorem}\label{semi_dispre}
Let $(u,p)$ be the solution of (\ref{onetwo})-(\ref{ic}). Then, there exists a positive constant $C$ such that the postprocessed pressure,
$\tilde p$, satisfies the following bounds:
\begin{eqnarray}\label{ptilde_mas}
\|p({t^*})-\tilde p({t^*})\|_{L^2/{\Bbb R}}\le \frac{C}{{t^*}^{(r-1)/2}}{h^{r}}|\log(h)|^{r'},
\end{eqnarray}
where $r'=2$ if $r=2$ and $r'=1$ if $r=3,4$.

\end{theorem}
\begin{proof}
The proof follows the same steps as \cite[Theorem~5.3]{jbj_regularity}. Applying the inf-sup condition (\ref{inf-sup})
it is easy to see that
\begin{eqnarray*}
\beta \|p({t^*})-\tilde p({t^*})\|_{L^2/{\Bbb R}}&\le& \|\tilde u({t^*})-u({t^*})\|_1+\|u_t({t^*})-\dot u_h({t^*})\|_{-1}
\\
&&\quad+\|F(u_h({t^*}),u_h({t^*}))-F(u({t^*}),u({t^*}))\|_{-1}.
\end{eqnarray*}
Applying now (\ref{utilde_lineal}) and (\ref{utilde_mas}) to bound the first term and reasoning as in the
proof of Theorem~\ref{semi_disve} to bound the other two terms we conclude (\ref{ptilde_mas}).
\end{proof}

\begin{remark}\label{remark31}
As a consequence of Theorems~\ref{semi_disve} and~\ref{semi_dispre} we obtain in the proof of Theorem~\ref{apos_semi1} that $(\tilde u-u_h)$ is an asymptotically exact
 estimator of the error $(u-u_h)$ while $(\tilde p-p_h)$ is an asymptotically exact  estimator of the error $(p-p_h)$. However, as we have
 already observed $\tilde u$ and $\tilde p$ are not computable in practice. In Theorems~\ref{apos_semi1}, 4 and \ref{th_pos_esti} we present
 different procedures to get computable error estimators.
\end{remark}

As we pointed out before the MFE approximations $(u_h,p_h)$ to the velocity and the pressure of the solution $(u,p)$ of the
evolutionary Navier-Stokes equations (\ref{onetwo})-(\ref{ic}) at any fixed time ${t^*}$ are also the approximations to the velocity and pressure
of the steady Stokes problem (\ref{posth0})-(\ref{posth1}). In Theorem~\ref{apos_semi1} we show that any a posteriori error estimator of the
error in the steady Stokes problem (\ref{posth0})-(\ref{posth1}) gives us an a posteriori indicator of the error in the
approximations to the evolutionary Navier-Stokes equations.

Using the notation of~\cite{kara-makri} we will denote in the sequel by $\xi_{\rm vel}((u_h,p_h),f,H^j)$, $j=0,1$, any a posteriori error estimator of the error $u_h-\tilde u$ in the
norm of $H^j(\Omega)^d$
in the approximation to the velocity in the steady Stokes problem (\ref{posth0})-(\ref{posth1}). We will denote by
$\xi_{\rm pres}((u_h,p_h),f,L^2/{\Bbb R})$ any error estimator of the quantity $\|p_h-\tilde p\|_{L^2/{\Bbb R}}$.

\begin{theorem}\label{apos_semi1}
Let $(u,p)$ be the solution of {\rm (\ref{onetwo})-(\ref{ic})} and
fix any positive time ${t^*}>0$. Assume that the Galerkin
approximation $(u_h,p_h)$ satisfies, for $h$ small enough and
$r=2,3,4$,
\begin{equation}\label{saturacion}
\|u({t^*})-u_h({t^*})\|_j\ge {C_r} h^{r-j},\quad j=0,1.
\end{equation}
for some positive constant $C_r=C_r(t^*)$.

(i)  If $\xi_{\rm vel}((u_h({t^*}),p_h({t^*})),f,H^j)$, $j=0,1$, is
an efficient error indicator of the error in the MFE approximation
to the steady Stokes problem  {\rm (\ref{posth0})-(\ref{posth1})}.
That is, if there exist positive constants, $C_1$ and $C_2$, that
are independent of the mesh size $h$, such that the following bound
holds
\begin{equation}\label{efficient_steady}
C_1\le \frac{\xi_{\rm vel}((u_h({t^*}),p_h({t^*})),f,H^j)}{\|\tilde
u({t^*})-u_h({t^*})\|_j}\le C_2,\quad j=0,1,
\end{equation}
 then $\xi_{\rm vel}((u_h({t^*}),p_h({t^*})),f,H^j)$, $j=0,1$, it
is also an efficient error indicator of the error in the MFE
approximation to the evolutionary Navier-Stokes equations, i.e.
there exist positive constants $C_3$ and $C_4$ that are independent
of the mesh size $h$ such that the following bound holds for $h$
small enough
\begin{equation}\label{efficient_evol}
C_3\le \frac{\xi_{\rm vel}((u_h({t^*}),p_h({t^*})),f,H^j)}{\|\
u({t^*})-u_h({t^*})\|_j}\le C_4,\quad j=0,1.
\end{equation}

(ii)  If $\xi_{\rm vel}((u_h({t^*}),p_h({t^*})),f,H^j)$, $j=0,1$ is
an asymptotically exact error estimator of the error in the steady
Stokes problem then it is also an asymptotically exact error
estimator of the error in the evolutionary Navier-Stokes equations.

(iii)  Analogous results are obtained in the approximations to the
pressure.

 In the case $r=2$, the results are valid only in the $H^1$ norm.

\end{theorem}
\begin{proof} For simplicity in the exposition we will concentrate on the cases $r=3,4$ in the approximations to the velocity, the
proof for the approximations to the pressure and for the case $r=2$ being the same
except for  obvious changes.

Let us first observe that
$$
\|u_h({t^*})-u({t^*})\|_j\le \|u_h({t^*})-\tilde u({t^*})\|_j+\|\tilde u({t^*})-u({t^*})\|_j,\quad j=0,1.
$$
Dividing by $\|u_h({t^*})-u({t^*})\|_j$, using (\ref{saturacion}) and applying
Theorem~\ref{semi_disve} we obtain
$$
1\le \frac{\|u_h({t^*})-\tilde u({t^*})\|_j}{\|u_h({t^*})-u({t^*})\|_j}+\frac{C {t^*}^{-((r-1)/2)}}{C_r}h|\log(h)|.
$$
Now, using (\ref{efficient_steady}) we get
\begin{eqnarray*}
\frac{\|u_h({t^*})-\tilde u({t^*})\|_j}{\|u_h({t^*})-u({t^*})\|_j}&=&\frac{\|u_h({t^*})-\tilde u({t^*})\|_j}{\|u_h({t^*})-u({t^*})\|_j}
\frac{\xi_{\rm vel}((u_h({t^*}),p_h({t^*})),f,H^j)}{\xi_{\rm vel}((u_h({t^*}),p_h({t^*})),f,H^j)}\nonumber\\
&\le& \frac{1}{C_1}\frac{\xi_{\rm vel}((u_h({t^*}),p_h({t^*})),f,H^j)}{\|u_h({t^*})-u({t^*})\|_j}.
\end{eqnarray*}
Taking $h$ small enough so that $\frac{C {t^*}^{-((r-1)/2)}}{C_r}h|\log(h)|\le {1/2},$  we get
\begin{equation}\label{primer}
\frac{C_1}{2}\le \frac{\xi_{\rm vel}((u_h({t^*}),p_h({t^*})),f,H^j)}{\|u_h({t^*})-u({t^*})\|_j}.
\end{equation}
Now, we use the decomposition
\begin{equation}\label{eq:asterisco}
\|u_h({t^*})-\tilde u({t^*})\|_j\le \|u_h({t^*})- u({t^*})\|_j+\| u({t^*})-\tilde u({t^*})\|_j,\quad j=0,1.
\end{equation}
Reasoning as before we get
$$
\frac{\|u_h({t^*})-\tilde u({t^*})\|_j}{\|u_h({t^*})-u({t^*})\|_j}\le 1+\frac{C {t^*}^{-((r-1)/2)}}{C_r}h|\log(h)|.
$$
Since
$$
\frac{\|u_h({t^*})-\tilde u({t^*})\|_j}{\|u_h({t^*})-u({t^*})\|_j}\ge
\frac{1}{C_2}\frac{\xi_{\rm vel}((u_h({t^*}),p_h({t^*})),f,H^j)}{\|u_h({t^*})-u({t^*})\|_j},
$$
we finally reach
\begin{equation}\label{segun}
\frac{\xi_{\rm vel}((u_h({t^*}),p_h({t^*})),f,H^j)}{\|u_h({t^*})-u({t^*})\|_j}\le \frac{3C_2}{2}.
\end{equation}

\noindent From (\ref{primer}) and (\ref{segun}) we conclude (\ref{efficient_evol}) with $C_3=C_1/2$ and $C_4=3C_2/2$.

Let us now assume that ${\xi_{\rm vel}((u_h({t^*}),p_h({t^*})),f,H^j)}$ is an asymptotically exact error estimator. Using again the decomposition~(\ref{eq:asterisco}) we have
$$
\lim_{h\rightarrow 0}\frac{\|u_h(t^*)-\tilde u(t^*)\|_j}{\|u_h(t^*)-u(t^*)\|_j}= 1
+\lim_{h\rightarrow 0}\frac{\|u(t^*)-\tilde u(t^*)\|_j}{\|u_h(t^*)-u(t^*)\|_j}
=1,
$$
the last equality being a consequence of~Theorem~(\ref{semi_disve}) and~the
saturation hypothesis~(\ref{saturacion}).
As we pointed out before, this limit implies that $(\tilde u-u_h)$
is an asymptotically exact estimator
of the error $(u-u_h)$.
Then
\begin{multline*}
\lim_{h\rightarrow 0}\frac{\xi_{\rm vel}((u_h({t^*}),p_h({t^*})),f,H^j)}{\|u_h(t^*)-u(t^*)\|_j}\\
=
\lim_{h\rightarrow 0}\frac{\xi_{\rm vel}((u_h({t^*}),p_h({t^*})),f,H^j)}{\|u_h(t^*)-\tilde u(t^*)\|_j}
\frac{\|u_h(t^*)-\tilde u(t^*)\|_j}{\|u_h(t^*)-u(t^*)\|_j}=1,
\end{multline*}
and ${\xi_{\rm vel}((u_h({t^*}),p_h({t^*})),f,H^j)}$ is also an asymptotically exact estimator of the error in the
approximation to the velocity of the evolutionary Navier-Stokes equations.
\end{proof}
\begin{remark}
We remark that with hypothesis (\ref{saturacion}) we are merely assuming that the term of order $h^{r-j}$ is really
present in the asymptotic expansion of the Galerkin error. Let us also notice that the constant $C_r$ in~(\ref{saturacion}) is, in general
$O({t^*}^{-(r-2)/2})$, so that the ratio ${t^*}^{-((r-1)/2)}/C_r$ in the proof
of~Theorem~\ref{apos_semi1} is, in general, $O({t^*}^{(-1/2)})$.
\end{remark}

The key point in Theorem~\ref{apos_semi1} comes from the observation that
if we decompose
\begin{equation}\label{decom_error}
u-u_h=(u-\tilde u)+(\tilde u-u_h),
\end{equation}
the first term on the right hand side of (\ref{decom_error}),
$u-\tilde u,$ is in general smaller, by a factor of size
$O(h\log(h))$, than the second one, $\tilde u-u_h$
(Theorem~\ref{semi_disve}). Then, to estimate the error $u-u_h$ we
can safely omit the term $u-\tilde u$ in (\ref{decom_error}).
Comparing with the analysis of \cite{kara-makri} for a nonstationary
linear Stokes model problem the main difference is that the two
terms in (\ref{decom_error}) are taken into account. In
Theorem~\ref{the_new} below we show that this kind of technique can
also be applied to the nonlinear Navier-Stokes equations.  The
advantage of this point of view is that hypothesis
(\ref{saturacion}) is not required for the proof of
Theorem~\ref{the_new}.
%Also, the regularity assumptions
%(\ref{eq:M_1})--(\ref{eq:u-int}) are not needed since the only
%regularity requirement  is that $\|\tilde u\|_2$ remains bounded and
%this bound can be reached from a priori bounds for the Galerkin
%approximation $u_h$.
Let us finally observe that  $(\dot u_h,\dot
p_h)$ are the MFE approximations to the solution $( \tilde
u_t,\tilde p_t)$ of the Stokes problem that we obtain deriving
respect to the time variable the Stokes problem
(\ref{posth0})-(\ref{posth1}). Then, we will denote by $\xi_{\rm
vel}((\dot u_h,\dot p_h),f_t,H^j)$, $j=-1,0,1$, any a posteriori
error estimator of the error $u_h-\tilde u_t$ in the norm of
$H^j(\Omega)^d$ in the approximation to the velocity of the
corresponding steady Stokes problem. The proof of the following
theorem follows the steps of the proof of \cite[Theorem
1]{Javier-Julia-estiman-2}.
\begin{theorem}\label{the_new}
Let $(u,p)$ be the solution of (\ref{onetwo})-(\ref{ic}) and let $(u_h,p_h)$ be its MFE Galerkin approximation. Then, the following
a posteriori error bound holds for $0\le t\le T$ and a constant $C$ independent of $h$.
\begin{eqnarray}\label{eq:new}
&&\|(u-u_h)(t)\|_0\le C\|u_0-u_h(0)\|_0+C\xi_{\rm vel}((u_h(0), p_h(0)),f(0),L^2)\qquad
\nonumber\\&&\ +\xi_{\rm vel}((u_h(t), p_h(t)),f(t),L^2)
 +Ct^{1/2}\max_{0\le s\le t}\xi_{\rm vel}((u_h, p_h),f,L^2)\\
 &&\ +C t^{1/2}\max_{0\le s\le t}\xi_{\rm vel}((\dot u_h, \dot p_h),f_s,H^{-1})\nonumber.
\end{eqnarray}
\end{theorem}
\begin{proof}
Let us denote by $\eta=u-\tilde u$. From (\ref{posth0})-(\ref{posth1}) it follows that
$$
\eta_t+A\eta+\Pi(F(u,u)-F(u_h,u_h))=\Pi(\dot u_h-\tilde u_t).
$$
Then $\eta$ satisfies the equation
\begin{eqnarray*}
&&\eta(t)=e^{-At}\eta(0)+\int_0^t e^{-A(t-s)}\Pi(F(\tilde u ,\tilde u)-F(u,u))~ds\nonumber\\
&&\quad + \int_0^t e^{-A(t-s)}\Pi(F(u_h,u_h)-F( \tilde u,\tilde u)~ds)+\int_0^t e^{-A(t-s)}\Pi(\dot u_h-\tilde u_t)~ds.
\end{eqnarray*}
Taking into account (\ref{lee1}) we get
\begin{eqnarray*}
\|\eta(t)\|_0\le \|\eta(0)\|_0+C\int_0^t\frac{\|A^{-1/2}\Pi(F(\tilde u,\tilde u)-F(u,u))\|_0}{\sqrt{t-s}}~ds&&
\nonumber\\
 +C\int_0^t \frac{\|A^{-1/2}\Pi(F(u_h,u_h)-F(\tilde u,\tilde u))\|_{0}}{\sqrt{t-s}}~ds
+C \int_0^t \frac{\|A^{-1/2}\Pi(\dot u_h-\tilde u_t)\|_0}{\sqrt{t-s}}~ds.&&
\end{eqnarray*}
We first observe that for any $v\in L^2(\Omega)^d$ we have $\|A^{-1/2}\Pi v\|_0\le C \|v\|_{-1}$. Then, taking into account
(\ref{eq:aux_bat}) we get
\begin{eqnarray*}
\|A^{-1/2}\Pi(F(\tilde u,\tilde u)-F(u,u))\|_0&\le& C\|\tilde u-u\|_0,\nonumber\\
\|A^{-1/2}\Pi(F( u_h, u_h)-F(\tilde u,\tilde u))\|_0&\le& C\|u_h-\tilde u \|_0.
\end{eqnarray*}
Let us observe that in order to apply (\ref{eq:aux_bat}) we require $u_h$ to satisfy (\ref{threshold}), which holds due to (\ref{eq:err_vel(t)}),
and $\|\tilde u\|_\infty$ and $\| \nabla \tilde u\|_{L^{2d/(d-1)}}$ to be bounded. Using (\ref{sob1}) both norms are
bounded in terms of $\|\tilde u\|_2$. Applying (\ref{eq:may06-nueva}) we get
\begin{eqnarray*}
\|\tilde u\|_2&\le& C\left(\|\dot u_h\|_0+\|u_h\cdot \nabla u_h\|_0\right)\nonumber\\
&\le&C\left(\|A_h u_h\|_0+\|\Pi_hF(u_h,u_h)\|_0 +\|\Pi_h f\|_0+\|u_h\cdot \nabla u_h\|_0\right).
\end{eqnarray*}
Finally, using that $\|A_h u_h\|_0$ is uniformly bounded, see
(\ref{A_h_uh}), and reasoning as in (\ref{eq:adelanto0}) to bound
the second and forth terms above we conclude $\|\tilde u\|_2$ is
uniformly bounded. Then, we arrive at
\begin{eqnarray*}
\|\eta(t)\|_0&\le& \|\eta(0)\|_0+C\int_0^t\frac{\|\eta(s)\|_0}{\sqrt{t-s}}~ds+C\int_0^t\frac{\|u_h(s)-\tilde u(s)\|_0}{\sqrt{t-s}}~ds
\nonumber\\
&&\quad + C\int_0^t\frac{\|\dot u_h(s)-\tilde u_s(s)\|_0}{\sqrt{t-s}}~ds.
\end{eqnarray*}
And then
\begin{eqnarray*}
\|\eta(t)\|_0&\le& \|\eta(0)\|_0+C\int_0^t\frac{\|\eta(s)\|_0}{\sqrt{t-s}}~ds+C t^{1/2}\max_{0\le s\le t}\xi_{\rm vel}((u_h,p_h),f,L^2)
\nonumber\\&&\quad+ C t^{1/2}\max_{0\le s\le t}\xi_{\rm vel}((\dot u_h,\dot p_h),f_s,L^2).
\end{eqnarray*}
A standard application of a generalized Gronwall lemma \cite{Henry} gives
\begin{eqnarray*}
\|\eta(t)\|_0&\le& C\|\eta(0)\|_0+C t^{1/2}\max_{0\le s\le t}\xi_{\rm vel}((u_h,p_h),f,L^2)\nonumber
\\&&\quad +C t^{1/2}\max_{0\le s\le t}\xi_{\rm vel}((\dot u_h,\dot p_h),f_s,L^2).
\end{eqnarray*}
Now, using decomposition (\ref{decom_error}) we conclude the proof.
\end{proof}
We observe that using the same proof, a similar bound for the $H^1(\Omega)^d$ norm of the error can be obtained changing only
$\xi_{\rm vel}(( u_h, p_h),f,L^2)$ by $\xi_{\rm vel}((u_h, p_h),f,H^1)$ and $\xi_{\rm vel}((\dot u_h,\dot p_h),f_t,H^{-1})$
by $\xi_{\rm vel}((\dot u_h,\dot p_h),f_t,L^2)$. Let us also remark that Theorem~\ref{the_new} allows to a posteriori obtain upper error bounds for the error in the approximation to the nonlinear Navier-Stokes equations using only upper
 error bounds for some Stokes problems depending only on the data and the computed approximation. However, the
estimation of the error at a time $t$ requires the estimation of the error of a family of Stokes problems with right hand
side depending on $\tau$, for all $\tau\in[0,t]$.

We now propose a simple procedure to estimate the error which is based on computing a MFE
 approximation
to the solution $(\tilde u(t^*),\tilde p(t^*))$ of
(\ref{posth0})-(\ref{posth1}) on a MFE space with better
approximation capabilities than $(X_{h,r},Q_{h,r-1})$ in which the
Galerkin approximation~$(u_h,p_h)$ is defined. This procedure was
applied to the $p$-version of the finite-element method for
evolutionary convection-reaction-diffusion equations in
\cite{Javier-Julia-estiman}. The main idea here is to use a second
approximation of different accuracy than that of the Galerkin
approximation of $(u,p)$ and whose computational cost hardly adds to
that of the Galerkin approximation itself.

Let us fix any time $t^*\in(0,T]$ and let approximate the solution $(\tilde u,\tilde p)$
of the Stokes problem (\ref{posth0})-(\ref{posth1})
by solving the following discrete Stokes problem:
find
$(\tilde{u}_{h}(t^*), \tilde{p}_{h}(t^*)) \in
\widetilde{X}\times\widetilde{Q}$ satisfying
\begin{eqnarray}\label{posth0dis}
     \left(\nabla\tilde{u}_{h}(t^*),
 \nabla \tilde{\phi}\right) +\left(\nabla \tilde{p}_{h}(t^*),
 \tilde{\phi} \right)
&=& \left(f, \tilde{\phi}\right) - \left(F(u_{h}(t^*),u_{h}(t^*)),\tilde{\phi}\right)\\\nonumber&& -
 \left( \dot u_{h}(t^*), \tilde{\phi}\right)  \quad \forall \, \tilde{\phi} \in \widetilde{X},\hspace*{-20pt}
  \\
\left( \nabla \cdot \tilde{u}_{h}(t^*),
\tilde{\psi} \right) &= & 0 \quad
\forall \, \tilde{\psi} \in \widetilde{Q}, \label{posth1dis}
\end{eqnarray}
where
$(\widetilde{X},\widetilde{Q})$ is either:
\begin{enumerate}
\item[(a)]
The same-order MFE over a finer grid. That is, for
$h'<h$, we choose
$(\widetilde{X},\widetilde{Q})=(X_{h',r}, Q_{h',r-1})$.
\item[(b)]
A higher-order MFE over the same grid. In this case we choose
$(\widetilde{X},$  $\widetilde{Q}) =(X_{h,r+1}, Q_{h,r})$.
\end{enumerate}
%In~\cite{jbj_regularity} the following result is proved.
%Although $(\tilde u_h, \tilde p_h)$ a the MFE approximation to the
%Stokes problem~(\ref{posth0})-(\ref{posth1}),
We now study the errors
$u-\tilde u_h$ and~$p-\tilde p_h$.

%thm 2.2
\begin{theorem}\label{teo4}
Let $(u,p)$ be  the solution  of {\rm (\ref{onetwo})--(\ref{ic})} and
for $r=2,3,4$,  and let {\rm(\ref{eq:u-inf})--(\ref{eq:u-int})} hold with $k=r+2$ Then,
there exists a positive constant~$C$ such that the postprocessed MFE
approximation to~$u$,
$\tilde{u}_{h}$ satisfies
the following bounds for $r=2,3,4$ and $t\in (0,T]$:

 (i)  if the postprocessing element is $(\widetilde{X},\widetilde{Q})=
(X_{h',r},Q_{h',r-1})$, then
\begin{align}
\|u(t)-\tilde{u}_{h}(t)\|_{j} &\le \frac{C}{t^{(r-2)/2}} (h')^{r-j}
 +
\frac{C}{t^{(r-1)/2}} h^{r+1-j}|\log{(h)}|^{r'},&&  j= 0,1 ,
\label{postcuad1}\\
\label{ulti1}
\|p(t)-\tilde{p}_{h}(t)\|_{L^2/{\mathbb R}}&\le
\frac{C}{t^{(r-2)/2}} (h')^{r-1}
+\frac{C}{t^{(r-1)/2}} h^{r}|\log{(h)}|^{r'},&&
\end{align}

 (ii) if the postprocessing element
is $(\widetilde{X},\widetilde{Q}) = (X_{h,r+1},Q_{h,r})$, then
\begin{align}
 \|u(t)-\tilde{u}_{h}(t)\|_{j} &\leq
 \frac{C}{t^{(r-1)/2}} h^{r+1-j}|\log{(h)}|^{r'},&& j=0,1,
 \label{postcuad2} \\
 \label{ulti2}
\|p(t)-\tilde{p}_{h}(t)\|_{L^2/{\mathbb R}}
       &\le \frac{C}{t^{(r-1)/2}} h^{r}|\log{(h)}|^{r'}.&&
\end{align}
For $r=2$ only the case $j=1$ in~(\ref{postcuad1}) and~(\ref{postcuad2}) holds.
In~(\ref{postcuad1})--(\ref{ulti2}), $r'=2$ when $r=2$ and $r'=1$ otherwise.
\end{theorem}
\begin{proof} The cases $r=3,4$ have been proven in Theorems~5.2 and~5.3
in~\cite{jbj_regularity}. Following the same arguments,
we now prove the results corresponding to $r=2$ and $(\widetilde{X},\widetilde{Q})=
(X_{h',r},Q_{h',r-1})$, the case $(\widetilde{X},\widetilde{Q})=
(X_{h,r+1},Q_{h,r})$ being similar, yet easier.
We decompose the error $u-\tilde u_h=(u-s_{h'}) + (s_{h'}-\tilde u_h)$, where
 $(s_{h'},q_{h'})\in X_{h',2}\times Q_{h',1}$ is the solution of
\begin{eqnarray}
\label{July-24th11}
 \left(\nabla s_{h'},
 \nabla {\phi}_{h'}\right)-\left( {q}_{h'},
 \nabla\cdot {\phi}_{h'} \right)
&=& \left(f-F(u,u)- u_t, {\phi}_{h'}\right) \ \forall \, {\phi}_{h'} \in {X_{h',2}},
  \\
\left( \nabla \cdot s_{h'},
{\psi}_{h'} \right) &= & 0 \quad
\forall \, {\psi}_{h'} \in Q_{h',1},
\label{July-24th12}
\end{eqnarray}
that is, $s_{h'}$ is the Stokes projection of~$u$ onto $V_{h'}$.
Since in view of~(\ref{stokespro})--(\ref{stokespre}) we have
$$
\left\| u-s_{h'}\right\|_1+\left\|p-q_{h'}\right\|_{L^2/{\mathbb R}}\le CM_{2}h',
$$
we only have to estimate $s_{h'}-\tilde u_{h}$ and $q_{h'}-\tilde p_h$. To do this,
we subtract~(\ref{posth0dis}) from~(\ref{July-24th11}),
 and take inner product with
$\tilde e_h=s_{h'}-\tilde u_h$ to get
$$
\left\|\nabla \tilde e_h\right\|_0^2\le
\bigl(\left\| u_t-\dot u_h\right\|_{-1} +\left\|F(u_h,u_h)-F(u,u)\right\|_{-1}\bigr)
\left\| \tilde e_h\right\|_1.
$$
Now applying~Lemma~\ref{le:z_t}, (\ref{eq:aux_bat}) and~(\ref{eq:err_vel(t)}) the proof
of~(\ref{postcuad1}) is finished.

To prove~(\ref{ulti1}), again we subtract~(\ref{posth0dis}) from (\ref{July-24th11}), rearrange
terms and apply the inf-sup condition~(\ref{inf-sup}) to get
$$
\beta\|q_{h'}-\tilde{p}_{h}\|_{L^2/{\mathbb R}}  \le
\left\|\nabla \tilde e_h\right\|_0 +
\left\| u_t-\dot u_h\right\|_{-1} +\left\|F(u_h,u_h)-F(u,u)\right\|_{-1}
$$
and the proof is finished with the same arguments used to
prove~(\ref{postcuad1}).\hfill \end{proof}

To estimate the error in $(u_h(t^*),p_h(t^*))$ we propose to take the difference between the postprocessed and the Galerkin approximations:
$$
\tilde \eta_{h,{\rm vel}}(t^*)=\tilde u_h(t^*)-u_h(t^*),\quad
\tilde \eta_{h,{\rm pres}}(t^*)=\tilde p_h(t^*)-p_h(t^*).
$$
In the following theorem we prove that this error estimator is efficient and asymptotically exact both in the
$L^2(\Omega)^d$ and $H^1(\Omega)^d$ norms and
it has the advantage of providing an improved approximation when added to the Galerkin MFE approximation.

\begin{theorem}\label{th_pos_esti}
Let $(u,p)$ be the solution of (\ref{onetwo})-(\ref{ic}) and
fix any positive time ${t^*}>0$. Assume that condition (\ref{saturacion}) is satisfied.
Then, there exist positive constants $h_0$, $\gamma_0<1$,
and  $C_1$, $C_2$,
$C_3$ and~$C_4$ such that, for $h<h_0$ and $0<\gamma<
\gamma_0$,
the error estimators $\tilde \eta_{h,{\rm vel}}(t^*)$
$\tilde \eta_{h,{\rm pres}}(t^*)$ satisfy the following bounds
when $(\widetilde X,\widetilde Q)=(X_{h',r},Q_{h',r-1})$ and $h'<\gamma h$:
\begin{equation}\label{efi_posdis}
C_1\le \frac{\|\tilde \eta_{h,{\rm vel}}(t^*)\|_j}{\|(u-u_h)(t^*)\|_j}\le C_2,\ j=0,1,\quad
C_3\le \frac{\|\tilde \eta_{h,{\rm pres}}(t^*)\|_{L^2/{\Bbb R}}}{\|(p-p_h)(t^*)\|_{L^2/{\Bbb R}}}\le C_4.
\end{equation}
Furthermore, if $(\widetilde X,\widetilde Q)=(X_{h',r},Q_{h',r-1})$, with $h'=h^{1+\epsilon}$, $\epsilon>0$, or
$(\widetilde X,\widetilde Q)=(X_{h,r+1},Q_{h,r})$ then
\begin{equation}\label{asin_posdis}
\lim_{h\rightarrow 0}\frac{\|\tilde \eta_{h,{\rm vel}}(t^*)\|_j}{\|(u-u_h)(t^*)\|_j}=1,\ j=0,1,\quad
\lim_{h\rightarrow 0}\frac{\|\tilde \eta_{h,{\rm pres}}(t^*)\|_{L^2/{\Bbb R}}}{\|(p-p_h)(t^*)\|_{L^2/{\Bbb R}}}=1.
\end{equation}
For the mini element, the case $j=0$ in~(\ref{efi_posdis}) and~(\ref{asin_posdis}) must be
excluded.
\end{theorem}
\begin{proof}
We will prove the estimates for the velocity in the case $r=3,4$,
since the estimates for the pressure and the case $r=2$ are obtained by
similar arguments but with obvious changes.
Let us observe that for $j=0,1$
\begin{eqnarray*}
\|u(t^*)-u_h(t^*)\|_j&\le& \|\tilde \eta_{h,{\rm vel}}(t^*)\|_j+\|\tilde u_h(t^*)-u(t^*)\|_j\nonumber\\
&\le&\|\tilde \eta_{h,{\rm vel}}(t^*)\|_j+\frac{C}{(t^*)^{(r-2)/2}}(h')^{r-j}\nonumber\\
&&\quad+\frac{C}{(t^*)^{(r-1)/2}}h^{r+1-j}|\log(h)|.
\end{eqnarray*}
On the other hand
\begin{eqnarray*}
\|\tilde \eta_{h,{\rm vel}}(t^*)\|_j&\le&\|u(t^*)-u_h(t^*)\|_j+\|\tilde u_h(t^*)-u(t^*)\|_j\nonumber\\
&\le&\|u(t^*)-u_h(t^*)\|_j+\frac{C}{(t^*)^{(r-2)/2}}(h')^{r-j}\nonumber\\
&&\quad+\frac{C}{(t^*)^{(r-1)/2}}h^{r+1-j}|\log(h)|.
\end{eqnarray*}
Using (\ref{saturacion}) we get
\begin{equation}\label{July-24th}
\left| \frac{\|\tilde \eta_{h,{\rm vel}}(t^*)\|_j}{\|(u-u_h)(t^*)\|_j}-1\right|\le\frac{C}{C_r}\left((t^*)^{-(r-2)/2}\left(\frac{h'}{h}\right)^{r-j}+(t^*)^{-(r-1)/2} |\log(h)| h\right).
\end{equation}
Taking $h'\le \gamma h$ and $h$ and~$\gamma$ sufficiently small, the bound (\ref{efi_posdis}) is readily obtained. The proof of (\ref{asin_posdis}) follows
straightforwardly from~(\ref{July-24th}), since in the case when
 $(\widetilde X,\widetilde Q)=(X_{h',r},Q_{h',r-1})$ with $h'=h^{1+\epsilon}$, $\epsilon>0$,
the term $(h'/h)^{r-j}\rightarrow 0$ when $h$ tends to zero, and in the case when
$(\widetilde X,\widetilde Q)=(X_{h,r+1},Q_{h,r})$ the term containing
the parameter $h'$ is not present.
\end{proof}

\section{A posteriori error estimations. Fully discrete case}
\label{sec:4}
In practice, it is not possible to compute the MFE approximation exactly, and, instead, some time-stepping procedure must be used
to approximate the solution of (\ref{ten})-(\ref{ten2}).
 Hence, for some time levels
$0=t_0<t_1<\cdots<t_N=T$, approximations $U_h^n\approx u_h(t_n)$ and
$P_h^n\approx p_h(t_n)$ are obtained. In this section we assume that
the approximations are obtained with the backward Euler method or
the two-step BDF which we now describe. For simplicity, we consider only
constant stepsizes, that is, for $N\ge 2$ integer, we fix
$k=T/N$, and we denote  $t_n=nk$, $n=0,1,\ldots,N$. For a sequence
$(y^n)_{n=0}^N$ we denote
$$
Dy^n=y^n-y^{n-1},\quad n=1,2\ldots,N.
$$
Given $U_h^{0}=u_h(0)$, a sequence $(U_h^n,P_h^n)$ of approximations to~$(u_h(t_n),p_h(t_n))$,
$n=1,\ldots N$, is obtained by means of the following recurrence relation:
\begin{align}\label{tend}
\left(d_tU_{h}^n, \phi_{h}\right)
&+     \left( \nabla U_{h}^n, \nabla \phi_{h}\right)\\
 &+
b\left(U_{h}^n, U_{h}^n, \phi_{h}\right) - \left( P_{h}^n, \nabla \cdot\phi_{h}\right) = (f,
\phi_{h}) \quad
\forall  \phi_{h} \in X_{h,r},\nonumber\\
&\left(\nabla \cdot U_{h}^n, \psi_{h}\right) = 0, \quad\forall \, \psi_{h} \in Q_{h,r-1},
\label{tend2}
\end{align}
where $d_t = k^{-1}D$ in the case\vspace*{-2pt} of the backward Euler method and
$d_t=k^{-1}(D+\frac{1}{2}D^2)$ for the two-step BDF. In\vspace*{-2pt} this last case, a second
starting value $U_h^{1}$ is needed. Here, we will always assume that $U_h^{1}$ is obtained by one
step of the backward Euler method.\vspace*{-1pt} Also, for both the backward Euler and the two-step
BDF, we assume that~$U_h^{0}=u_h(0)$, which is usually the case in practical situations.

We now define the time-discrete postprocessed approximation. Given
an approximation $d_t^* U_h^n$ to $\dot u_h(t_n)$,
 the time-discrete postprocessed velocity and pressure $(\widetilde U^n,\widetilde P^n)$  are defined as the solution of the following
Stokes problem:
\begin{align}\label{posth0n}
     \left(\nabla\tilde{U}^n,
 \nabla {\phi} \right) +\left(\nabla \tilde{P}^n,
 {\phi} \right)
&= \left(f, {\phi}\right) - b\left(U_{h}^n,U_{h}^n,{\phi}\right) -
 \left( d_t^*U_{h}^n, {\phi}\right),
 \ \forall \, {\phi} \in H_0^1(\Omega)^d,
  \\
\left( \nabla \cdot \tilde{U}^n, {\psi} \right) &=  0, \ \forall \, {\psi} \in L^2(\Omega)/{\Bbb
R}. \label{posth1n}
\end{align}
% Notice that
%$\tilde U^n\in  V$ and it
%satisfies
%\begin{equation}
%\label{posth0Pd}
%\tilde A\tilde U^n=
% \Pi\left(f-F\left(U_h^n,U_h^n\right)-d_t^* U_h^n\right).
%\end{equation}
For reasons already analyzed in~\cite{Javier-Julia-estiman-3} and \cite{jbj_fully}  we define
\begin{equation}
\label{dtu} d_t^* U_h^n=\Pi_h f -A_hU_h^n- \Pi_h F\left(U_h^n,U_h^n\right)
\end{equation}
as an adequate approximation to the time derivative $\dot u_h(t_n)$.

For the analysis of the errors $u(t)-\tilde U^n$ and $p(t)-\tilde P^n$ we follow~\cite{jbj_fully},
where the MFE approximations to the Stokes problem~(\ref{posth0n})--(\ref{posth1n}) are analyzed.
We start by decomposing the errors $u(t)-\tilde U^n$ and $p(t)-\tilde P^n$ as follows,
\begin{eqnarray}
\label{eq:decomtilde} u(t_n)-\tilde U^n&= &(u(t)-\tilde u(t_n))+\tilde e^n,
\\
\label{eq:decomtildeP} p(t_n)-\tilde P^n&=&(p(t_n)-\tilde p(t_n))+\tilde \pi^n,
\end{eqnarray}
where $\tilde e^n=\tilde u(t_n)-\tilde U^n$ and $\tilde \pi^n=\tilde p(t_n)-\tilde P^n$ are the
temporal errors of the time-discrete postprocessed velocity and pressure $(\tilde U^n, \tilde
P^n)$. The first terms on the right-hand sides of~(\ref{eq:decomtilde})--(\ref{eq:decomtildeP}) are
the errors of the postprocessed approximation that were studied in the previous section.

Let us denote by $e_h^n=u_h(t_n)-U_h^n$, the temporal error of the
MFE approximation to the velocity, and by $\pi_h^n=p_h(t_n)-P_h^n$,
the temporal error of the MFE approximation to the pressure. In the
present section we bound $(\tilde e^n-e_h^n)$ and $(\tilde
\pi^n-\pi_h^n)$ in terms of $e_h^n$.
%We start with
%Lemma~\ref{lema_need} bellow, whose proof is similar to that of
%\cite[Proposition 3.1]{jbj_fully}.

The error bounds in the following lemma are similar to
         those of~\cite[Proposition~3.1]{jbj_fully} where
         error estimates for  MFE approximations of the Stokes
         problem~(\ref{posth0n})--(\ref{posth1n}) are obtained.
\begin{lemma}\label{lema_need}There exists a positive constant $C=C(\max_{0\le t\le T}\|A_h u_h(t)\|_0)$ such that
\begin{eqnarray}
\qquad\quad\|\tilde e^n-e_h^n\|_j &\le& C h^{2-j}\bigl ( \|e_h^n\|_1+\|e_h^n\|_1^3 +
\|A_he_h^n\|_0\bigr ),\  j=0,1,\quad 1\le n\le N.
\label{error_post_fully}\\
\|\tilde  \pi^n-\pi_h^n\|_{L^2/{\Bbb R}}&\le& C h\bigl ( \|e_h^n\|_1+\|e_h^n\|_1^3 +
\|A_he_h^n\|_0\bigr ),\quad 1\le n\le N. \label{error_post_fully_pre}
\end{eqnarray}
\end{lemma}
\begin{proof}
Let us denote by $l=g+(d_t^* U_h^n-\dot u_h(t_n))$ where $g=F(U_h^n,U_h^n)-F(u_h(t_n),u_h(t_n))$.
Subtracting~(\ref{posth0n})--(\ref{posth1n}) from~(\ref{posth0})--(\ref{posth1}) we have that the
temporal errors~$(\tilde e^n,\tilde \pi^n)$ of the time-discrete postprocessed velocity and pressure are the
solution of the following Stokes problem
\begin{eqnarray}\label{otro_stokes}
(\nabla \tilde e^n,\nabla \phi)+(\nabla \tilde \pi^n,\phi)&=&(l,\phi),\quad\forall \phi\in
H_0^1(\Omega)^d,
\\
(\nabla\cdot\tilde e^n,\psi)&=&0,\quad\forall \psi\in L^2(\Omega)/{\Bbb R}.\label{otro_stokes2}
\end{eqnarray}
On the other hand, subtracting (\ref{tend})-(\ref{tend2}) from (\ref{ten})-(\ref{ten2}) and taking
into account that, thanks to definition (\ref{dtu}) $d_t U_h^n=d_t^* U_h^n$, we get that the
temporal errors~$(e_h^n,\pi_h^n)$ of the fully discrete MFE approximation satisfy
\begin{align*}
(\nabla  e_h^n,\nabla \phi_h)+(\nabla \pi_h^n,\phi_h)&=(l,\phi_h),\quad\forall \phi_h\in X_{h,r},
\\
(\nabla\cdot e_h^n,\psi_h)&=0,\quad\forall \psi_h\in Q_{h,r-1},
\end{align*}
and thus $(e_h^n,\pi_h^n)$ is the MFE approximation to the solution $(\tilde e^n,\tilde \pi^n)$ of
(\ref{otro_stokes})--(\ref{otro_stokes2}). Using then (\ref{stokespro+1}) we get
$$
\|\tilde e^n-e_h^n\|_j\le C h^{2-j}\|l\|_0.
$$
For the pressure we apply (\ref{stokespre}) and (\ref{preHR}) to obtain
$$
\|\tilde \pi^n-\pi_h^n\|_{L^2/{\Bbb R}}\le C h \|\tilde \pi^n\|_{H^1/{\Bbb R}}\le C h\|l\|_0.
$$
Then, to conclude, it only remains to bound $\|l\|_0$. From the definition of $d^*_t U_h^n$ it is
easy to see that
%\begin{equation}\label{siroco}
$$
d_t^* U_h^n- \dot u_h(t_n)=A_he_h^n
-\Pi_h\left(F\left(U_h^n,U_h^n\right)-F(u_h(t_n),u_h(t_n))\right),
%\end{equation}
%In view of~(\ref{siroco}) we have
$$
so that
$$
\|d_t^* U_h^n- \dot u_h(t_n)\|_0\le \|A_he_h^n\|_0+\|g\|_0.
$$
Now, by writing~$g$ as
$$
%\begin{equation}
%\label{eq:lag}
g=F(e_h^n,u_h(t_n))+F(u_h(t_n),e_h^n)-F(e_h^n,e_h^n),
%\end{equation}
$$
and using (\ref{eq:adelanto0})--(\ref{eq:adelanto1}) we get
$$
 \|g\|_0\le
\left (\|A_hu_h(t_n) \|_0 \|e_h^n\|_1+
\|e_h^n\|_1^{3/2}\|A_he_h^n\|_0^{1/2}\right),
$$
from which we finally conclude (\ref{error_post_fully}) and (\ref{error_post_fully_pre}).
\hfill
\end{proof}

Let us consider the quantities $\widetilde U^n-U_h^n$ and
$\widetilde P^n-P_h^n$ as a posteriori indicators of the error in
the fully discrete approximations to the velocity and pressure
respectively. Then, we obtain the following result:
\begin{theorem}\label{th_esti_time_d}
Let $(u,p)$ be the solution of (\ref{onetwo})--(\ref{ic}) and let
(\ref{tildeM2}) hold. Assume that the fully discrete MFE
approximations $(U_h^n,P_h^n)$, $n=0\ldots,N=T/k$ are obtained by
the backward Euler method or the two-step BDF
(\ref{tend})--(\ref{tend2}), and let $(\widetilde U^n,\widetilde
P^n)$ be the solution of~(\ref{posth0n})--(\ref{posth1n}). Then, for
$n=1,\ldots,N$,
\begin{eqnarray}\label{decom_prin2}
\|\widetilde U^{n}-U_h^n\|_j&\le& \|\tilde u(t_n)-u_h(t_n)\|_j+
C'_{l_0} h^{2-j}\frac{k^{l_0}}{t_n^{l_0}},\qquad j=0,1,
\\
\label{decom_prin2p}
\|\widetilde P^{n}-P_h^n\|_{L^2/{\Bbb R}}&\le& \|\tilde p(t_n)-p_h(t_n)\|_{L^2/{\Bbb
R}}+C'_{l_0}
h\frac{k^{l_0}}{t_n^{l_0}},
\end{eqnarray}
where $C'_{l_0}$ is the constant in~(\ref{pri_fully})--(\ref{pri_fully_pres}),
$l_0=1$ for the backward Euler method and $l_0=2$ for the two-step BDF.
\end{theorem}
\begin{proof}
In \cite[Theorems~5.4 and~5.7]{jbj_fully} we prove that
if~(\ref{tildeM2}) and the case $l=2$ in~(\ref{stokespro}) hold, the
errors $e_h^n$ of these two time integration procedures satisfy for
$k$ small enough that
\begin{equation}
\label{eq:orden1} \|e_h^n\|_0+t_n\|A_he_h^n\|_0 \le {\cal C}_{l_0}
\frac{k^{l_0}}{{t_n}^{l_0-1}},\quad 1\le n\le N,
\end{equation}
for a certain constants~${\cal C}_1$ and~$C_2$, where $l_0=1$ for
the backward Euler method and $l_0=2$ for the two-step BDF. Since
 $\|A_h^{1/2}e_h^n\|_0 \le \|e_h^n\|_0^{1/2}
\|A_h e_h^n\|_0^{1/2}$, and then $\|e_h^n\|_1\le C \|e_h^n\|_0^{1/2}
\|A_h e_h^n\|_0^{1/2}$, from (\ref{eq:orden1}) and
(\ref{error_post_fully})-(\ref{error_post_fully_pre}) we finally
reach that for $k$ small enough
\begin{eqnarray}\label{pri_fully}
\|\tilde e^n-e_h^n\|_j&\le& C'_{l_0}
h^{2-j}\frac{k^{l_0}}{t_n^{l_0}},\ j=0,1,\ 1\le n\le N,
\\
\label{pri_fully_pres} \|\tilde \pi^n-\pi_h^n\|_{L^2/{\Bbb R}}&\le&
C'_{l_0} h\frac{k^{l_0}}{t_n^{l_0}},\ 1\le n\le N,
\end{eqnarray}
where $C'_{l_0}$ is a positive constant.

Let us decompose the estimators as follows:
\begin{eqnarray}\label{decom_prin}
\widetilde U^{n}-U_h^n&=& \left( \widetilde U^{n}-\tilde
u(t_n)\right)+\left(\tilde
u(t_n)-u_h(t_n)\right)+\left(u_h(t_n)-U_h^n\right)\nonumber
\\&=&\left(\tilde u(t_n)-u_h(t_n)\right)+(e_h^n-\tilde e^n),
\\
\label{decom_prinp} \widetilde P^{n}-P_h^n &=& \left( \widetilde
P^{n}-\tilde p(t_n)\right)+\left(\tilde
p(t_n)-p_h(t_n)\right)+\left(p_h(t_n)-P_h^n\right)\nonumber
\\&=&\left(\tilde p(t_n)-p_h(t_n)\right)+(\pi_h^n-\tilde \pi^n),
\end{eqnarray}
which implies
\begin{eqnarray*}
\|\widetilde U^{n}-U_h^n\|_j&\le& \|\tilde
u(t_n)-u_h(t_n)\|_j+\|e_h^n-\tilde e^n\|_j, \qquad j=1,2,
\\
\|\widetilde P^{n}-P_h^n\|_{L^2/{\Bbb R}}&\le& \|\tilde
p(t_n)-p_h(t_n)\|_{L^2/{\Bbb R}}+\|\pi_h^n-\tilde \pi^n\|_{L^2/{\Bbb
R}}.
\end{eqnarray*}
Thus, in view of~(\ref{pri_fully})--(\ref{pri_fully_pres}) we obtain
(\ref{decom_prin2}) and (\ref{decom_prin2p})
\end{proof}

Let us comment on the practical implications of this theorem.
Observe that from (\ref{decom_prin}) and (\ref{decom_prinp}) the
fully discrete estimators $\widetilde U^n-U_h^n$ and $\widetilde
P^n-P_h^n$ can be  both decomposed as the sum of two terms. The
first one is the semi-discrete a posteriori error estimator we have
studied in the previous section (see Remark~\ref{remark31}) and
which we showed it is an asymptotically exact estimator of the
spatial error of $U_h^n$ and $P_h^n$ respectively. On the other
hand, as shown in (\ref{pri_fully})--(\ref{pri_fully_pres}),  the
size of the second term is in asymptotically smaller than the
temporal error of $U_h^n$ and $P_h^n$ respectively.  We conclude
that, as long as the spatial an temporal errors are not too
unbalanced (i.e., they are not of very different sizes), the first
term in~(\ref{decom_prin}) and (\ref{decom_prinp}) is dominant and
then the quantities $\widetilde U^{n}-U_h^n$ and $\widetilde
P^n-P_h^n$ are a posteriori error estimators of the spatial error of
the fully discrete approximations to the velocity and pressure
respectively. The control of the temporal error can be then
accomplished by standard and well-stablished techniques in the field
of numerical integration of ordinary differential equations.

Now, we remark that  $(\widetilde U^n,\widetilde P^n)$ are obviously not computable.  However, we
observe that the fully discrete approximation $(U_h^n,P_h^n)$ of the evolutionary Navier-Stokes
equation is also the approximation to the Stokes problem (\ref{posth0n})-(\ref{posth1n}) whose
solution is $(\widetilde U^{n},\widetilde P^{n})$. Then, one can use any of the available error
estimators for a steady Stokes problem to estimate the quantities $\|\widetilde U^{n}-U_h^n\|_j$
and $\|\widetilde P^{n}-P_h^n\|_{L^2/{\Bbb R}}$, which, as we have already proved, are error
indicators of the spatial errors of the fully discrete approximations to the velocity and pressure,
respectively.

To conclude, we show a procedure to get computable estimates of the error in the fully discrete
approximations. We define the fully discrete postprocessed approximation
$(\widetilde U_h^n,\widetilde P_h^n)$ as the solution of the following Stokes problem (see
\cite{jbj_fully}):
\begin{align}\label{posth0ndis}
     \left(\nabla\widetilde{U}_{h}^n,
 \nabla \tilde{\phi} \right) +\left(\nabla \widetilde{P}_{h}^n,
 \tilde{\phi} \right)
&= \left(f, \tilde{\phi}\right) - b\left(U_{h}^n,U_{h}^n,\tilde{\phi}\right) -
 \left( d_t^*U_{h}^n, \tilde{\phi}\right)
 \quad \forall \, \tilde{\phi} \in \widetilde{X},
  \\
\left( \nabla \cdot \widetilde{U}_{h}^n, \tilde{\psi} \right) &=  0 \quad \forall \, \tilde{\psi} \in
\widetilde{Q}, \label{posth1ndis}
\end{align}
where $(\widetilde{X},\widetilde{Q})$ is as in (\ref{posth0dis})-(\ref{posth1dis}).

Let us denote
by $\tilde e_h^n=\tilde u_h(t_n)-\widetilde U_h^n$ and $\tilde \pi_h^n=\tilde p_h(t_n)-\widetilde
P_h^n$ the temporal errors of the fully discrete postprocessed approximation~$(\tilde U_h^n, \tilde
P_h^n)$ (observe that the semi-discrete postprocessed approximation $(\tilde u_h,\tilde p_h)$ is
defined in (\ref{posth0dis})-(\ref{posth1dis})). Let us denote, as before, by $e_h^n$ the temporal
error of the MFE approximation to the velocity. Then, we have the following bounds.
\begin{lemma}
\label{prop:err_post_fully}
There exists
a positive constant $C=C(\max_{0\le t\le T}\|A_hu_h(t)\|_0)$
such that for $1\le n\le N$ the following bounds hold
\begin{eqnarray}
\qquad\qquad\quad \|\tilde e_h^n-e_h^n\|_j &\le& C h^{2-j}\bigl (
\|e_h^n\|_1+\|e_h^n\|_1^3 + \|A_he_h^n\|_0\bigr ), \ j=0,1,
\label{error_post_fully_vel}
\\
\label{duda2pp}
 \|\tilde \pi_h^n-\pi_h^n\|_{L^2(\Omega)/\mathbb R}
&\le& C h\bigl (
\|e_h^n\|_1+\|e_h^n\|_1^3 + \|A_he_h^n\|_0\bigr ).
\end{eqnarray}
\end{lemma}
\begin{proof}
The bound (\ref{error_post_fully_vel}) is proved in \cite[Proposition 3.1]{jbj_fully}.
To prove (\ref{duda2pp}) we decompose
$$
\|\tilde \pi_h^n-\pi_h^n\|_{L^2(\Omega)/\mathbb R}\le \|\tilde \pi_h^n-\tilde
\pi^n\|_{L^2(\Omega)/\mathbb R} +\|\tilde \pi^n-\pi_h^n\|_{L^2(\Omega)/\mathbb R}.
$$
The second term above is bounded in (\ref{error_post_fully_pre}) of Lemma~\ref{lema_need}. For the
first we observe that $\tilde \pi_h^n$ is the MFE approximation in $\widetilde Q$ to the pressure
$\tilde \pi^n$ in (\ref{otro_stokes})-(\ref{otro_stokes2}) so that the same reasoning used in the
proof of  Lemma~\ref{lema_need} allow us to obtain
$$
\|\tilde \pi_h^n-\tilde \pi^n\|_{L^2(\Omega)/\mathbb R}\le C  h\|\tilde \pi^n\|_{H^1/{\Bbb R}}\le C
h\bigl ( \|e_h^n\|_1+\|e_h^n\|_1^3 + \|A_he_h^n\|_0\bigr ).
$$
\end{proof}

Using (\ref{eq:orden1}) as before, we get the analogous to (\ref{pri_fully}) and (\ref{pri_fully_pres}), i.e.,  for $k$ small enough
the following bound holds
\begin{eqnarray}\label{pri_fully_dis}
\|\tilde e_h^n-e_h^n\|_j&\le& C'_{l_0} h^{2-j}\frac{k^{l_0}}{t_n^{l_0}},\ j=0,1,\ 1\le n\le N,\\
\label{pri_fully_dis_pre}
\|\tilde \pi_h^n-\pi_h^n\|_{L^2/{\Bbb R}}&\le& C'_{l_0} h\frac{k^{l_0}}{t_n^{l_0}},\ 1\le n\le N.
\end{eqnarray}
where $C'_{l_0}$ is a positive constant.

Similarly to~(\ref{decom_prin})--(\ref{decom_prinp}) we write
$\widetilde U_h^{n}-U_h^n=\bigl(\tilde u_h(t_n)-u_h(t_n)\bigr)+
\bigl(e_h^n-\tilde e_h^n\bigr)$
and $\widetilde P_h^{n}-P_h^n=\bigl(\tilde p_h(t_n)-p_h(t_n)\bigr)+
\bigl(\pi_h^n-\tilde \pi_h^n\bigr)$,
so that in view of~(\ref{pri_fully_dis})--(\ref{pri_fully_dis_pre})
we have the following result.

\begin{theorem}\label{th_esti_fully_d}
Let $(u,p)$ be the solution of (\ref{onetwo})--(\ref{ic}) and let
(\ref{tildeM2}) hold. Assume that the fully discrete MFE
approximations $(U_h^n,P_h^n)$, $n=0\ldots,N=T/k$ are obtained by
the backward Euler method or the two-step BDF
(\ref{tend})--(\ref{tend2}), and let $(\widetilde U_h^n,\widetilde
P_h^n)$ be the solution of~(\ref{posth0ndis})--(\ref{posth1ndis}).
Then, for $n=1,\ldots,N$,
\begin{eqnarray}\label{decom_prin2_dis}
\|\widetilde U_h^{n}-U_h^n\|_j&\le& \|\tilde u_h(t_n)-u_h(t_n)\|_j+
C'_{l_0} h^{2-j}\frac{k^{l_0}}{t_n^{l_0}},\qquad j=0,1,
\\
\label{decom_prin2_disp}
\|\widetilde P_h^{n}-P_h^n\|_{L^2/{\Bbb R}}&\le& \|\tilde p_h(t_n)-p_h(t_n)\|_{L^2/{\Bbb
R}}+C'_{l_0}
h\frac{k^{l_0}}{t_n^{l_0}},
\end{eqnarray}
where $C'_{l_0}$ is the constant in~(\ref{pri_fully_dis})--(\ref{pri_fully_dis_pre}),
$l_0=1$ for the backward Euler method and $l_0=2$ for the two-step BDF.
\end{theorem}

The practical implications of this result are similar to~those
of~Theorem~\ref{th_esti_time_d}, that is,
the first term on the right-hand side of~(\ref{decom_prin2_dis})
is an error indicator of the
spatial error (see Theorem~\ref{th_pos_esti}) while the second one is
asymptotically smaller than the temporal error. As a consequence,
the quantity $(\widetilde U_h^n-U_h^n)$ is a computable estimator of
the spatial error of the fully discrete velocity $U_h^n$ whenever
the temporal and spatial errors of $U_h^n$ are more or less of
the same size. As before, similar arguments apply for the pressure.
We  remark that having balanced spatial and temporal errors in the fully discrete approximation is the more common case in  practical computations since
one usually looks for a final solution with small total error.

As in the semi-discrete case, the advantage of these error estimators is that
they produce  enhanced (in space) approximations when
they are added to the Galerkin MFE approximations.

\section{Numerical experiments}
We consider the equations
\begin{eqnarray}
\label{onetwo_nu}
u_t -\nu\Delta u + (u\cdot\nabla)u + \nabla p &=& f,\\
\bet
{\rm div}(u)&=&0,\nonumber
\end{eqnarray}
in the domain $\Omega=[0,1]\times[0,1]$ subject to homogeneous Dirichlet boundary conditions. For the
numerical experiments of this section we approximate the equations using the mini-element \cite{Brezzi-Fortin91} over
a regular triangulation of $\Omega$ induced by the set of nodes $(i/N,j/N)$, $0\le i,j\le N$, where $N=1/h$ is an
integer. For the time integration we use the two-step BDF method with fixed time step. For the first step we apply the backward
Euler method.
%The non-linear systems are
%solved iteratively using fixed point iteration in which we approximate the non-linear term $(u^{m+1}\cdot \nabla) u^{m+1}$
%by $(u^{m}\cdot \nabla)u^{m+1}$.
In the first numerical experiment we study the semi-discrete in space case. To this end
in the numerical experiments we integrate in time with a time-step small enough in order to have negligible temporal
errors. We take the forcing term $f(t,x)$ such that the solution of (\ref{onetwo_nu}) with $\nu=0.05$ is
\begin{eqnarray}\label{solu_fix}
u^1(x,y,t)&=&2\pi \varphi(t)  \sin^2(\pi x)\sin(\pi y)\cos(\pi y),\nonumber\\
u^2(x,y,t)&=&-2\pi \varphi(t)\sin^2(\pi y)\sin(\pi x)\cos(\pi x),\\
p(x,y,t)&=&20 \varphi(t) x^2 y\nonumber.
\end{eqnarray}
We chose $\varphi(t)=t$ in the first numerical experiment.

When using the mini-element it has been observed and reported in the
literature (see for instance \cite{Verfurth1}, \cite{VerfurthSIAM}, \cite{Bank-Welfert2} \cite{Kim},
\cite{Pier1} and \cite{Pier2}) that the linear part of the approximation to the velocity, $u_h^l$,
is a better approximation to the solution $u$ than $u_h$ itself. The bubble part of the approximation
is only introduced for stability reasons and does not improve the approximation to the velocity
and pressure terms. For this reason in the numerical experiments of this section we only consider
the errors in the linear approximation to the velocity. Also, following \cite{bbj}, we postprocess
only the linear approximation to the velocity, i.e., we solve the Stokes problem (\ref{posth0dis})-(\ref{posth1dis})
with $u_h^l$ and $\dot u_h^l$ on the right-hand-side instead of $u_h$ and $\dot u_h$. The finite element space
at the postprocessed step is the same mini-element defined over a refined mesh of size $h'$. We show the Galerkin
errors and the a posteriori error estimates obtained at time $t^*=0.5$ by taking the difference between the
postprocessed and the standard approximations to the velocity and the pressure. In Figure~\ref{semi}, we
\begin{figure}[h]
\includegraphics[width=6cm]{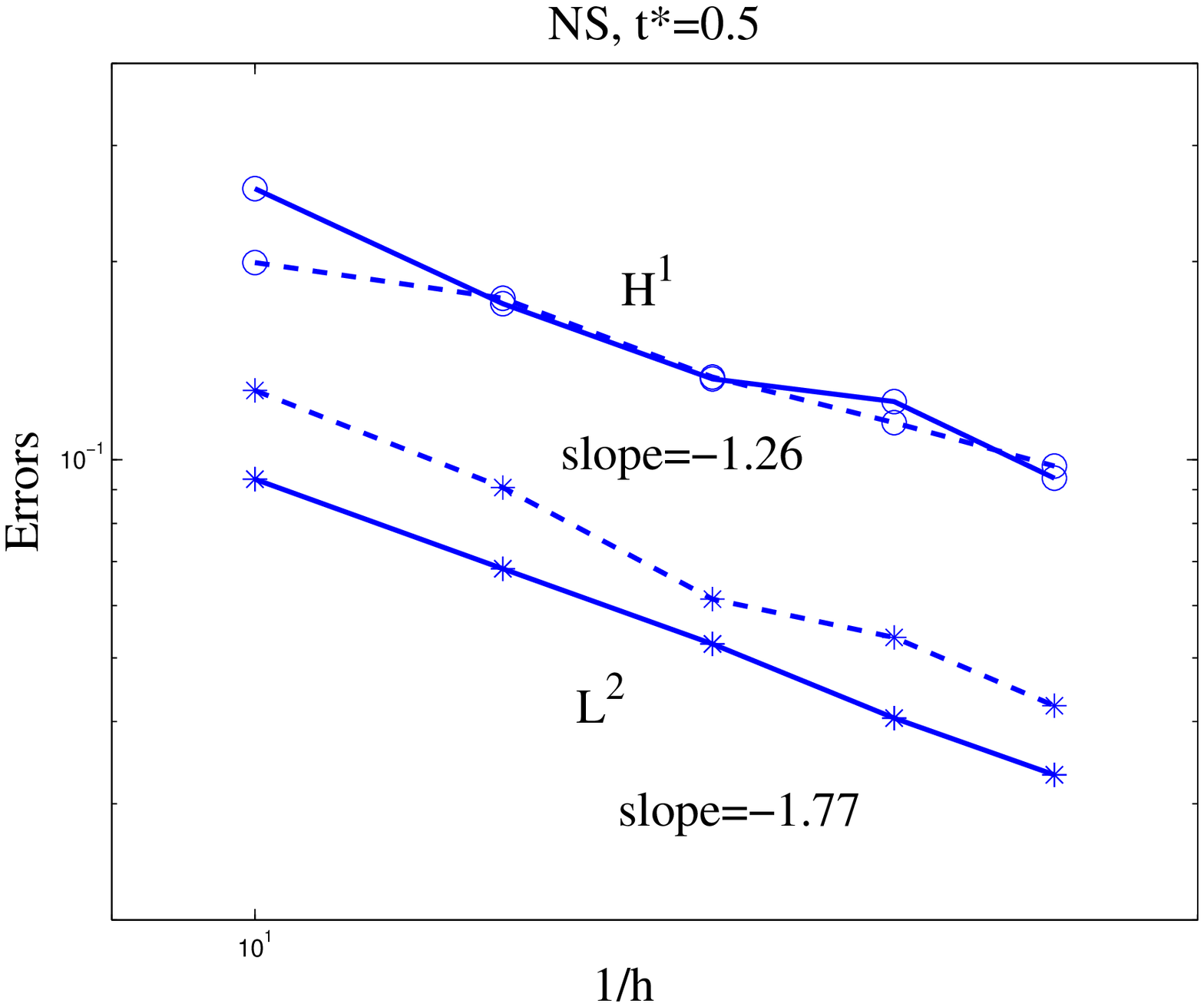}
\hspace{-0.5cm} 
\mbox{} \hfill
\includegraphics[width=6cm]{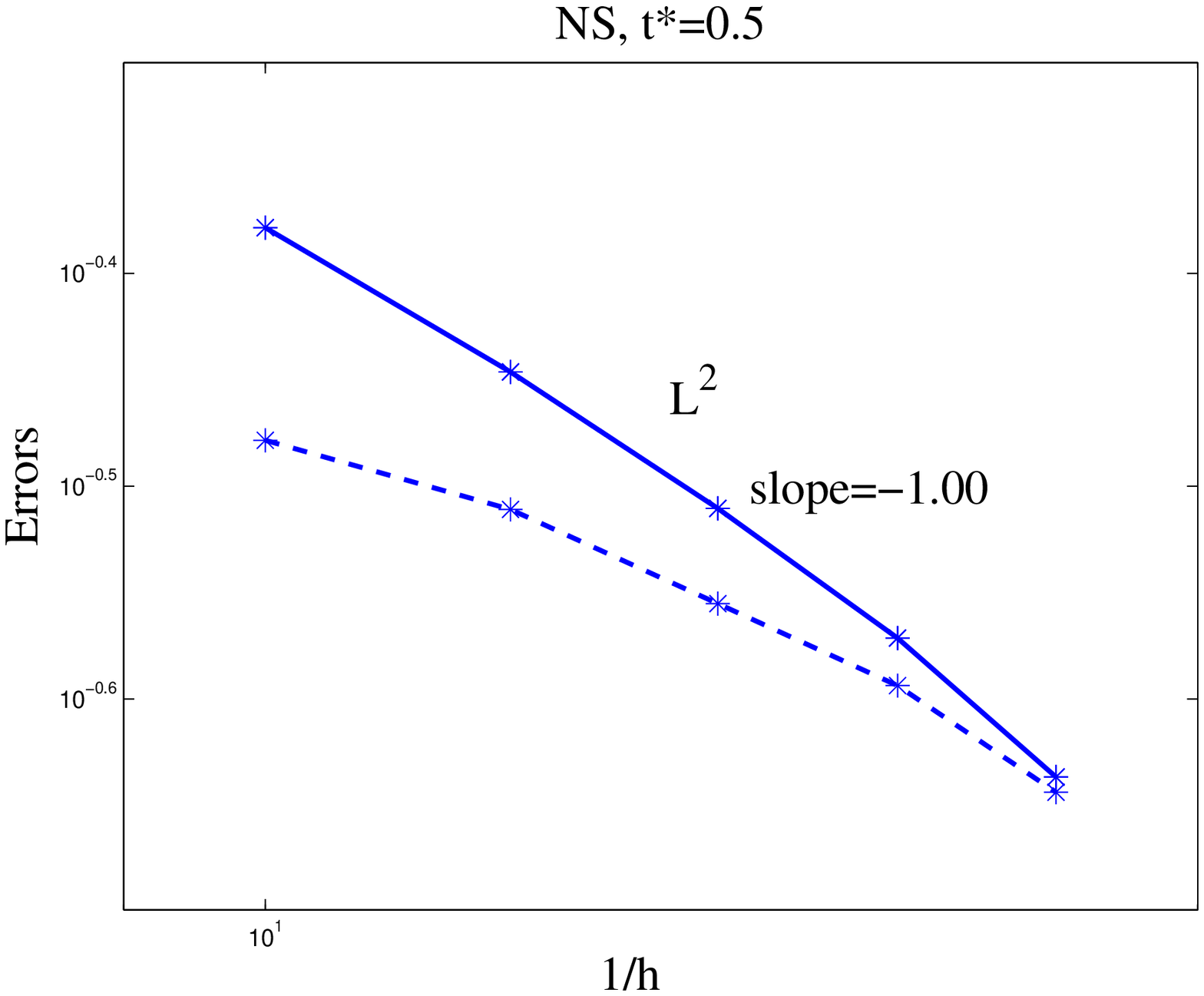}
\caption{Errors (solid lines) and estimations (dashed lines)
in $L^2$ (asterisks) and $H^1$ (circles) for $h=1/10$, $1/12$, $1/14$, $1/16$ and $1/18$ and
$h'=1/24$, $1/30$, $1/34$, $1/38$ and $1/40$ respectively. On the left, error estimations for the first
component of the velocity. On the right, error estimations for the pressure.}\label{semi}
\end{figure}
have represented the errors in the first component of the velocity of the Galerkin approximation in the $L^2$ and $H^1$ norms and the errors
for the pressure in the $L^2$ norm
using solid lines. We have used dashed lines to represent the error estimations. The results for the second component
of the velocity are completely analogous and they are not reported here. The $L^2$ errors of the pressure, on the
right of Figure~\ref{semi}, are approximately twice as those of the $H^1$ errors of the velocity, on the left of Figure~\ref{semi}, in this example.
We can observe that with the procedure we propose in this paper we get very accurate estimations of the errors, specially in
the $H^1$ norm of the velocity. The difference between the behavior of the error estimations in the $L^2$ and $H^1$ norms of the velocity are
due to the fact that for first order approximations the postprocessed procedure increases the rate
of convergence of the standard method only in the $H^1$ norm for the velocity and the $L^2$ norm for the pressure. However, since
the postprocessed method  produces smaller errors than the Galerkin method also in the $L^2$ norm it can also be used to estimate
the errors in this norm, as it can be checked in the experiment. On the right of Figure~\ref{semi} we can clearly observe the asymptotically
exact behavior of the estimator  in the $L^2$ errors in the pressure in agreement with (\ref{asin_posdis}) of Theorem~\ref{th_pos_esti}.

Let us denote by
$$
\theta_{\rm vel}=\frac{\tilde u_h^1(t^*)-u_h^1(t^*)}{u^1(t^*)-u_h^1(t^*)},\quad \theta_{\rm pre}=\frac{\tilde p_h(t^*)-p(t^*)}{p(t^*)-p_h(t^*)},
$$
the efficiency indexes for the first component of the velocity and for the pressure. In Table~\ref{tabla1} we have represented the values
of the $L^2$ and $H^1$ norms of the velocity index and the $L^2/{\Bbb R}$ norm of the pressure index for the experiments in Figure~\ref{semi}.
\begin{table}[h]
\begin{center}
\begin{tabular}{|c|c|c|c|}
\hline
 $h$& $\|\theta_{\rm vel}\|_0$&$\|\theta_{\rm vel}\|_1$& $\|\theta_{\rm pre}\|_{L^2/{\Bbb R}}$\\
\hline
1/10&$1.3640$  & $0.7721$&$1.2588$\\
 \hline
1/12&$1.3280$& $1.0197$&$1.1602$ \\
 \hline
1/14&$1.1695$ &$1.0068$&$1.1084$ \\
 \hline
1/16&$1.3259$ &$0.9290$&$1.0526$ \\
\hline
1/18&$1.2741$ &$1.0438$&$1.0167$\\
\hline
\end{tabular}
\end{center} \caption{Efficiency indexes}\label{tabla1}
\end{table}
We deduce again from the values of the efficiency indexes that the a posteriori error estimates are very accurate, all the values are remarkably
 close to $1$, which is the optimal value for the efficiency index. More precisely, we can observe that the values of the efficiency index in the $L^2$ norm for the
velocity in this experiment belong to the interval $[1.1695,1.3640]$. The values in the $H^1$ norm for the
velocity lie on the interval $[0.7721,1.0438]$ and, finally, the values for the pressure are in the interval
$[1.0167,1.2588]$.

To conclude, we show a numerical experiment to check the behavior of the estimators in the fully discrete case. We choose the forcing term $f$
such that the solution of (\ref{onetwo_nu}) is (\ref{solu_fix}) with $\varphi(t)=\sin((2\pi+\pi/2)t)$. The value of $\nu=0.05$ and
the final time $t^*=0.5$ are the same as before.
\begin{figure}[h]
\hspace{-0.5cm} 
\includegraphics[width=6cm]{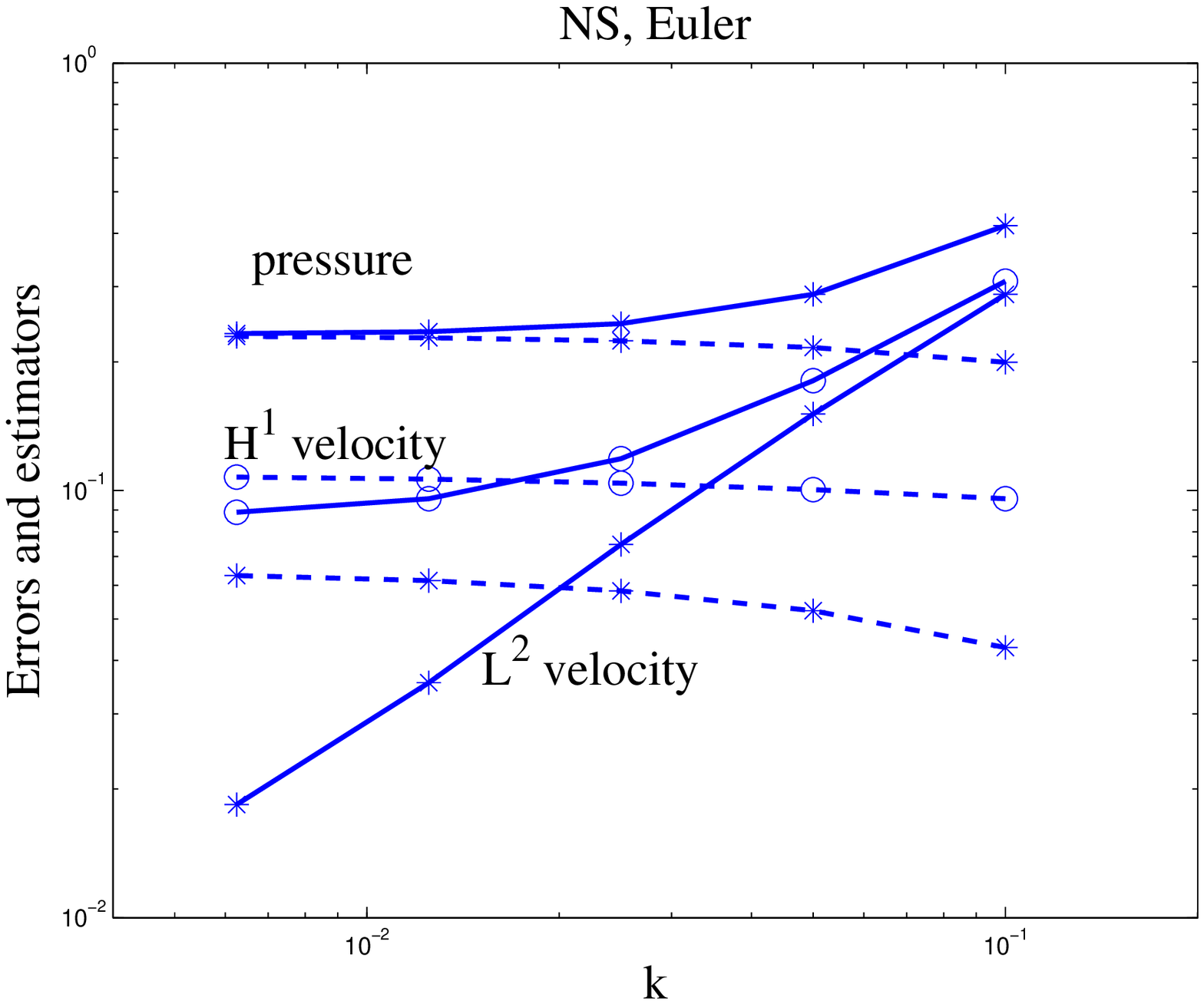} 
\mbox{} \hfill
\includegraphics[width=6cm]{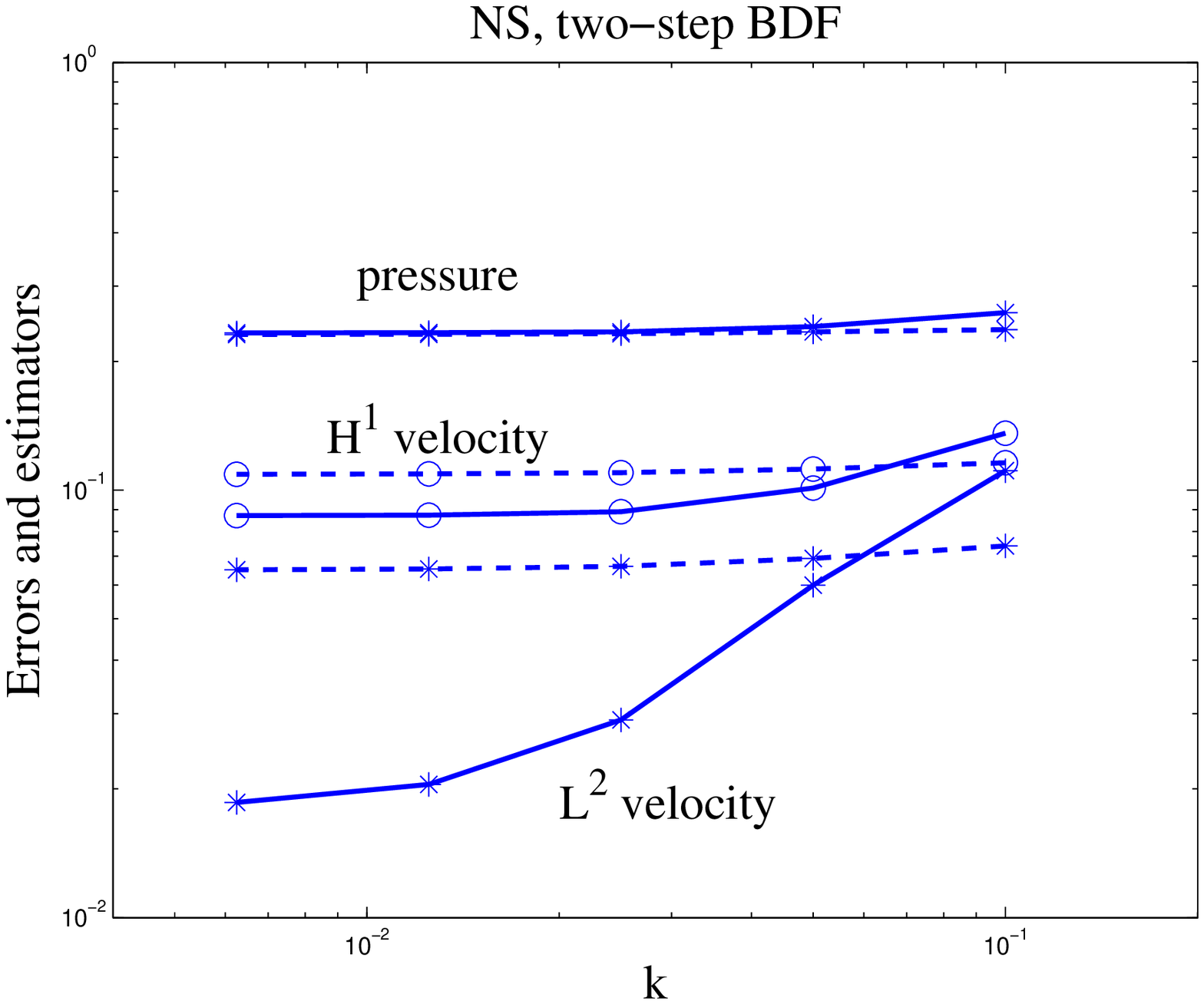}
\caption{Errors (solid lines) and
estimations (dashed lines) in $L^2$ (asterisks) and $H^1$ (circles)
for $h=1/18$. On the left: Euler; on the right: two-step BDF for
$k=1/10$ to $k=1/160$.}\label{ulti}
\end{figure}
In Figure~\ref{ulti}, on the left, we have represented the errors obtained using the implicit
Euler method as a time integrator for different values of the fixed time step $k$ ranging from $k=1/10$ to $k=1/160$ halving each time the value of $k$.
For the spacial discretization we use the mini-element with always the same value of $h=1/18$. We
use solid lines for the errors in the Galerkin method and dashed lines for the estimations, as before. The $L^2$ norm errors are marked with asterisks while
the $H^1$ norm errors are marked with circles. We estimate the errors using the postprocessed method computed with the same mini-element
over a refined mesh of size $h'=1/40$. We observe that the Galerkin errors decrease as $k$ decreases until a value that corresponds to the spatial error of the approximation. On the contrary, the error estimations lie on an almost horizontal line, both for the velocity in the $L^2$ and
$H^1$ norms and for the pressure. This means, as we stated in Section 4.2,
that the error estimations we propose are a measure
of the spatial errors, even when the errors in the Galerkin method are polluted by errors coming from the temporal discretization. In this experiment
the error estimations are very accurate for the spatial errors of the velocity in the $H^1$ norm and for the errors in the pressure.
As commented above, the fact that postprocessing linear elements does not increase the convergence rate
 in the $L^2$ norm is reflected in the precision of the error estimations in the $L^2$ norm. On the right of Figure~\ref{ulti} we have represented the
errors obtained when we integrate in time with the two-step BDF and
fixed time step. The only remarkable difference is that, as we
expected from the second order rate of convergence of the method in
time, the temporal errors are smaller for the same values of the
fixed time step $k$. Again, the estimations lie on a horizontal line
being essentially the same as in the experiment on the left.

\end{document}